\documentclass{amsart}
\usepackage{amssymb,amsmath,amsfonts}

\def\Xint#1{\mathchoice
   {\XXint\displaystyle\textstyle{#1}}
   {\XXint\textstyle\scriptstyle{#1}}
   {\XXint\scriptstyle\scriptscriptstyle{#1}}
   {\XXint\scriptscriptstyle\scriptscriptstyle{#1}}
   \!\int}
\def\XXint#1#2#3{{\setbox0=\hbox{$#1{#2#3}{\int}$}
     \vcenter{\hbox{$#2#3$}}\kern-.5\wd0}}

\def\dashint{\Xint-}

\newtheorem{theorem}{Theorem}[section]
\newtheorem{lemma}[theorem]{Lemma}
\newtheorem{corollary}[theorem]{Corollary}

\theoremstyle{definition}

\theoremstyle{remark}
\newtheorem{remark}[theorem]{Remark}

\newcommand{\mysection}[1]{\section{#1}
\setcounter{equation}{0}}

\newcommand{\bR}{\mathbb R}

\newcommand\cL{\mathcal{L}}

\newcommand\wto{\rightharpoonup}
\newcommand\rV{\mathring{V}}
\newcommand\rW{\mathring{W}}
\newcommand\cLt{{}^t\!\mathcal{L}}
\newcommand\U{\mathcal{U}}
\newcommand{\VMO}{\mathrm{VMO}}
\newcommand{\PH}{\mathrm{(PH)}}

\newcommand{\ip}[1]{\left\langle#1\right\rangle}
\newcommand{\set}[1]{\left\{#1\right\}}
\newcommand{\norm}[1]{\left\lVert#1\right\rVert}
\newcommand{\abs}[1]{\left\lvert#1\right\rvert}
\newcommand{\tri}[1]{|\!|\!|#1|\!|\!|}

\renewcommand{\epsilon}{\varepsilon}
\renewcommand{\vec}[1]{\boldsymbol{#1}}

\DeclareMathOperator*{\esssup}{ess\,sup}
\DeclareMathOperator{\dist}{dist}

\begin{document}
\title[Green's matrices]
{On the Green's matrices of strongly parabolic systems of second order}

\author[S. Cho]{Sungwon Cho}
\address[S. Cho]{Department of Mathematics,
Michigan State University, East Lansing, MI 48824, USA}
\email{cho@math.msu.edu}

\author[H. Dong]{Hongjie Dong}
\address[H. Dong]
{School of Mathematics, Institute for Advanced Study,
Einstein Drive, Princeton, NJ 08540, USA}
\email{hjdong@math.ias.edu}

\author[S. Kim]{Seick Kim}
\address[S. Kim]{Centre for Mathematics and its Applications,
The Australian National University, ACT 0200, Australia}
\email{seick.kim@maths.anu.edu.au}

\thanks{The second named author is partially supported by the
National Science Foundation under agreement No. DMS-0111298.}

\thanks{The third named author is supported by the Australian Research Council.}

\subjclass[2000]{Primary 35A08, 35K40; Secondary 35B45}

\keywords{Green function, Green's matrix, fundamental solution, fundamental matrix, second order parabolic system, Gaussian estimate.}

\begin{abstract}
We establish existence and various estimates of fundamental matrices
and Green's matrices for divergence form, second order strongly parabolic
systems in arbitrary cylindrical domains under the assumption that
solutions of the systems satisfy an interior H\"older continuity estimate.
We present a unified approach valid for both the scalar and the vectorial cases.
\end{abstract}

\maketitle

\mysection{Introduction} \label{intro}
In this paper, we study Green's matrices (or Green's functions) of second
order, strongly parabolic systems of divergence form
\begin{equation*}
\tag{P}\label{P}
\sum_{j=1}^N \cL_{ij}u^j:=
D_t u^i- \sum_{j=1}^N \sum_{\alpha,\beta=1}^n
D_{x_\alpha}(A^{\alpha\beta}_{ij}(t,x) D_{x_\beta}u^j),\quad i=1,\ldots,N,
\end{equation*}
in a cylindrical domain
$\U=\bR\times \Omega$ where $\Omega$ is an open connected set in $\bR^n$
with $n\ge 1$.
By a Green's matrix for the system \eqref{P} in $\U=\bR\times\Omega$
we mean an $N\times N$ matrix valued function
$\Gamma_{ij}(t,x,s,y)$  ($x,y\in\Omega$ and $t,s\in\bR$) which satisfies the following
(see Theorem~\ref{thm1} for more precise statement):
\begin{align*}
\sum_{j=1}^N \cL_{ij}\Gamma_{jk}(\cdot,\cdot,s,y)&=\delta_{ik}\delta_{(s,y)}
(\cdot,\cdot),\\
\lim_{t\to s+}
\Gamma_{ij}(t,x,s,\cdot)&=\delta_{ij}\delta_{x}
(\cdot),\\
\Gamma_{ij}(\cdot,\cdot,s,y)=0\quad&\text{on}\quad
\partial\U=\bR\times\partial\Omega,
\end{align*}
where $\delta_{ij}$ is the Kronecker delta symbol,
$\delta_{(s,y)}(\cdot,\cdot)$ and $\delta_{x}(\cdot)$ are
Dirac delta functions. 
In particular, when $\U=\bR^{n+1}$, the Green's matrix (or Green's function)
is called the fundamental matrix (or fundamental solution).

We prove that if weak solutions of \eqref{P} satisfy
an interior H\"older continuity estimate
(see Section~\ref{sec:PH} for the precise formulation),
then the Green's matrix exists in $\U$ and satisfies several natural
growth properties; see Theorem~\ref{thm1}.
Moreover, when $\U=\bR^{n+1}$, we show that the fundamental matrix
satisfies the semi-group property as well as the upper Gaussian bound of
Aronson \cite{Aronson}; see Theorem~\ref{thm2}.
The method we use does not involve Harnack type inequalities or the maximum
principle, and works for both the scalar and the vectorial situation.
Moreover, we do not require the base $\Omega$ of the cylinder
$\U=\bR\times\Omega$ to be bounded or to have a regular boundary.
In the scalar case (i.e., $N=1$), such an interior H\"older continuity estimate
for weak solutions is a direct consequence of 
the celebrated theorem of Nash \cite{Nash}
(see also Moser \cite{Moser}),
and thus, we in particular prove the existence and growth properties
of Green's functions for the uniformly parabolic equations of divergence
form with bounded measurable coefficients in arbitrary cylindrical domains;
see Corollary~\ref{cor1}.
If $n=2$ (i.e., $\Omega\subset \bR^2$)
and the coefficients $A^{\alpha\beta}_{ij}(t,x)$
of the system \eqref{P} are independent of $t$,
then we will find that the Green's matrix exists in any cylindrical domain
and satisfies the growth properties stated in Theorem~\ref{thm1};
see Corollary~\ref{cor2}.
Also, if $N>1$ and the coefficients $A^{\alpha\beta}_{ij}(t,x)$
of the system \eqref{P} are uniformly continuous in $x$
and measurable in $t$ or belong to the class of $\VMO$,
then we have the same conclusion; see Corollary~\ref{cor3}.

Let us briefly review the history of work in this area.
Fundamental solutions of parabolic equations of divergence form with bounded
measurable coefficients have been a subject of research for many years.
The first significant step in this direction was made by Nash \cite{Nash}
who established certain estimates of the fundamental solutions in proving
local H\"older continuity of weak solutions.
Aronson \cite{Aronson} proved Gaussian upper and lower bounds for
the fundamental solutions by using the parabolic Harnack inequality
of Moser \cite{Moser}.
Fabes and Strook \cite{FS} showed that the idea of Nash could be
used to establish Aronson's Gaussian bounds, which consequently gave
a new proof of Moser's parabolic Harnack inequality.
In \cite{Auscher}, Auscher gave a new proof of Aronson's Gaussian upper bound
for the fundamental solution of parabolic equations with time independent
coefficients, which carries over to the case of a complex perturbation
of real coefficients. 
Recently, it is noted in \cite{HofKim2} that Aronson's upper bound
is equivalent to the local boundedness property of weak solutions of strongly
parabolic systems.
Green's functions of elliptic equations of divergence form in bounded domains
have been extensively studied by
Littman, Stampacchia, and Weinberger \cite{LSU}
and Gr\"{u}ter and Widman \cite{GW},
whereas the Green's matrices of the elliptic systems with continuous
coefficients in bounded $C^1$ domains have been discussed by
Fuchs \cite{Fuchs} and Dolzmann and M\"uller \cite{DM}.
Very recently, Hofmann and Kim \cite{HofKim1} gave a unified approach in
studying Green's functions/matrices in arbitrary domains
valid for both scalar equations and systems of elliptic type by
considering a class of operators $L$ such that weak solutions of $Lu=0$
satisfy an interior H\"older estimate.
Some parts of the present article may be considered as a natural follow-up
of their work in the parabolic setting.
Readers interested in the construction of fundamental matrices
for parabolic systems in nondivergence form with H\"older continuous
coefficients are asked to refer to Eidel'man \cite{Eidelman},
Friedman \cite{Friedman}, or Lady\v{z}enskaja et al. \cite{LSU}.
Also, we would like to bring attention to a paper by Escauriaza
\cite{Escauriaza} on the fundamental solutions of elliptic and parabolic
equations in nondivergence form.

The organization of this paper is as follows.
In Section~\ref{main}, we define the property (PH) for parabolic
systems in terms of Morrey-Campanato type H\"older estimates
for weak solutions.
As we mentioned earlier, in the scalar case,
such a property is a direct consequence of the interior H\"older continuity
estimates by Nash \cite{Nash} and Moser \cite{Moser}.
We also present some other examples of parabolic systems
satisfying the property (PH), including the \textit{almost diagonal systems}
and the systems with coefficients from the class $\VMO_x$ (see below
for the definitions).
We close the section by presenting our main theorems.
In Theorem~\ref{thm1}, we state the existence, uniqueness, and properties of
Green's matrices in an arbitrary cylindrical domain for
the parabolic systems satisfying the property (PH).
Theorem~\ref{thm2} establishes Gaussian bounds and semi-group properties
for the fundamental matrices of parabolic systems satisfying the property (PH).
In Section~\ref{sec3}, we prove Theorem~\ref{thm1} in the entire space
assuming that the property (PH) holds globally.
In Section~\ref{sec4}, we give the proof for Theorem~\ref{thm1} in the general
setting.
Finally, we prove Theorem~\ref{thm2} in Section~\ref{sec5}.

\mysection{Preliminaries and main results} \label{main}

\subsection{Basic Notations}
We use $X=(t,x)$ to denote a point in
$\bR^{n+1}=\bR\times \bR^n$ with $n\ge 1$;
$x=(x_1,\ldots, x_n)$ will always be a point in $\bR^n$.
We also write $Y=(s,y)$, $X_0=(t_0,x_0)$, etc.
We define the parabolic distance between the points $X=(t,x)$ and $Y=(s,y)$
in $\bR^{n+1}$ as
\begin{equation*}
|X-Y|_p:=\max(\sqrt{\abs{t-s}},\abs{x-y}),
\end{equation*}
where $\abs{\,\cdot\,}$ denotes the usual Euclidean norm.
We usually use $\U$ to denote an open set in $\bR^{n+1}$ and $\Omega$
to denote an open set in $\bR^n$.
We define
\begin{equation*}
B_r(x)=\set{y\in\bR^n:|y-x|<r}
\end{equation*}
and use the following notations for basic cylinders in $\bR^{n+1}$:
\begin{align*}
Q^-_r(X)&=(t-r^2,t)\times B_r(x),\\
Q^+_r(X)&=(t,t+r^2)\times B_r(x),\\
Q_r(X)&=(t-r^2,t+r^2)\times B_r(x).
\end{align*}
Note that $Q_r(X)=\set{Y\in\bR^{n+1}:|Y-X|_p<r}$.
We use $\partial_p Q^-_r(X)$ and $\partial_p Q^+_r(X)$
to denote the parabolic forward and backward
boundaries of $Q^{-}_r(X)$ and $Q^{+}_r(X)$, respectively; i.e.,
\begin{align*}
\partial_p Q^-_r(X)&=(t-r^2,t)\times \partial B_r(x) \cup \{t-r^2\}\times
\overline B_r(x),\\
\partial_p Q^+_r(X)&=(t,t+r^2)\times \partial B_r(x) \cup \{t+r^2\}\times
\overline B_r(x),
\end{align*}
where  $\partial B_r(x)$ and $\overline B_r(x)$ denote
the usual boundary and closure of $B_r(x)$ in $\bR^n$;
i.e., $\partial B_r(x)=\set{y\in\bR^n:|y-x|=r}$
and $\overline B_r(x)=\set{y\in\bR^n:|y-x|\leq r}$.

For a given function $u=u(X)=u(t,x)$, we use
$D_i u$ for $\partial u/\partial x_i$
while we use $u_t$ (or sometimes $D_t u$) for $\partial u/\partial t$.
We also write $Du$ (or sometimes $D_x u$) for the vector
$(D_1 u,\ldots,D_n u)$.
If $u$ is a function in a set $Q\subset\bR^{n+1}$, we denote
\begin{equation*}
[u]_{C^{\mu,\mu/2}(Q)}:=\sup_{\substack{X\neq Y\\ X,Y\in Q}}
\frac{|u(X)-u(Y)|}{|X-Y|_p^\mu},\quad\text{where }\mu\in(0,1].
\end{equation*}
For a Lebesgue measurable set $Q\subset \bR^{n+1}$ (resp. $S\subset \bR^n$),
we write $\abs{Q}$ (resp. $\abs{S}$)
for the Lebesgue measure of the set $Q$ (resp. $S$).
For $u\in L^1(Q)$ (resp. $u\in L^1(S)$)
we use the notation $\dashint_Q u=\frac{1}{\abs{Q}}\int_Q u$
(resp. $\dashint_S u = \frac{1}{\abs{S}} \int_S u$).

\subsection{Function Spaces}
We follow the notation of \cite{LSU} with a slight variation.
For $\U\subset\bR^{n+1}$, we write $\U(t_0)$ for the
set of all points $(t_0,x)$ in $\U$ and $I(\U)$ for
the set of all $t$ such that $\U(t)$ is nonempty.
We denote
\begin{equation*}
\tri{u}_{\U}^2= \|Du\|_{L^{2}(\U)}^2
+\esssup\limits_{t\in I(\U)}\|u(t,\cdot)\|_{L^2(\U(t))}^2.
\end{equation*}
For $\vec{u}=(u^1,\ldots,u^N)$, we write
$\tri{\vec{u}}_{\U}^2 := \sum_{i=1}^N \tri{u^i}_{\U}^2$ so that
\begin{equation*}
\tri{\vec{u}}_{\U}^2 = \|D\vec{u}\|_{L^{2}(\U)}^2
+\esssup\limits_{t\in I(\U)}\|\vec{u}(t,\cdot)\|_{L^2(\U(t))}^2.
\end{equation*}
In the rest of this subsection, we shall denote by $Q$ the cylinder
$(a,b)\times\Omega$,
where $-\infty<a<b<\infty$ and $\Omega$ is an open connected
(possibly unbounded) set in $\bR^n$.
We denote by $W^{1,0}_2(Q)$ the Hilbert space with the inner product
\begin{equation*}
\ip{u,v}_{W^{1,0}_2(Q)}:=\int_Q uv+\sum_{k=1}^n \int_Q D_k u D_k v
\end{equation*}
and by $W^{1,1}_2(Q)$ the Hilbert space with the inner product
\begin{equation*}
\ip{u,v}_{W^{1,1}_2(Q)}:=\int_Q uv+\sum_{k=1}^n \int_Q D_k u D_k v
+\int_Q u_t v_t.
\end{equation*}
We define $V_2(Q)$ as the Banach space consisting of all elements
of $W^{1,0}_2(Q)$ having a finite norm $\norm{u}_{V_2(Q)}:= \tri{u}_{Q}$
and $V^{1,0}_2(Q)$ as the Banach space consisting of all elements
of $V_2(Q)$ that are continuous in $t$ in the norm of $L^2(\Omega)$, with
the norm
\begin{equation*}  
\tri{u}_{Q} = \left(\norm{D u}_{L^2(Q)}^2
+\max\limits_{a\leq t\leq b}\norm{u(t,\cdot)}_{L^2(\Omega)}^2\right)^{1/2}.
\end{equation*}
The continuity in $t$ of a function $u(t,x)$ in the norm of $L^2(\Omega)$
means that 
\begin{equation*}
\lim\limits_{h\to 0}\norm{u(t+h,\cdot)-u(t,\cdot)}_{L^2(\Omega)}=0.
\end{equation*}
The space $V^{1,0}_2(Q)$ is obtained by completing
the set $W^{1,1}_2(Q)$ in the norm of $V_2(Q)$.

We write $C^\infty_{c,p}(Q)$ for the set of all functions
$u\in C^\infty(\overline Q)$
with compact supports in $[a,b]\times\Omega$
while $C^\infty_c(\U)$ denotes the set of all infinitely differentiable
functions with compact supports in $\U$.
We denote by $\rW^{1,0}_2(Q)$ and $\rW^{1,1}_2(Q)$ the closure
of $C^\infty_{c,p}(Q)$ in the Hilbert spaces $W^{1,0}_2(Q)$ and
$W^{1,1}_2(Q)$, respectively.
We define $\rV_2(Q):=V_2(Q)\cap \rW^{1,0}_2(Q)$ and
$\rV^{1,0}_2(Q)=V^{1,0}_2(Q)\cap \rW^{1,0}_2(Q)$.
It is routine to check that $\rV_2(Q)$ and $\rV^{1,0}_2(Q)$ are
subspaces of the Banach spaces $V_2(Q)$ and $V^{1,0}_2(Q)$, respectively.
By a well known embedding theorem (see e.g., \cite[\S II.3]{LSU}),
we have
\begin{equation} \label{eqn:2.2}
\norm{u}_{L^{2+4/n}(Q)} \leq C(n) \tri{u}_{Q}
\quad\forall u\in \rV_2(Q).
\end{equation}
When $\U$ is an infinite cylinder (i.e., $(a,\infty)\times\Omega$,
$(-\infty,b)\times\Omega$, or $(-\infty,\infty)\times\Omega$),
we say that $u\in V_2(\U)$ if $u\in V_2(\U_T)$ for all $T>0$,
where $\U_T:=\U\cap(-T,T)\times\bR^n$, and
$\tri{u}_{\U}<\infty$.
Similarly, we say that $u\in \rV_2(\U)$
(resp. $V^{1,0}_2(\U)$, $\rV^{1,0}_2(\U)$)
if $u\in \rV_2(\U_T)$
(resp. $V^{1,0}_2(\U_T)$, $\rV^{1,0}_2(\U_T)$) for all $T>0$
and $\tri{u}_\U<\infty$.
Also, we write $u\in C^\infty_{c,p}(\U)$ if $u\in C^\infty_{c,p}(\U_T)$
for all $T>0$.
For a sequence $\set{u_k}_{k=1}^\infty$ in $V_2(\U)$, we say that
$u_k \wto u$ ``very weakly'' in $V_2(\U)$ for some $u\in V_2(\U)$ if
$u_k \wto u$ weakly in $W^{1,0}_2(\U_T)$ for all $T>0$.
If $\set{u_k}_{k=1}^\infty$ is a bounded sequence in $V_2(\U)$
(resp. $\rV_2(\U)$), then there exists a subsequence
$\{u_{k_j}\}_{j=1}^\infty\subseteq \{u_k\}_{k=1}^\infty$
and $u\in V_2(\U)$ (resp. $\rV_2(\U)$)
such that $u_{k_j}\wto u$ ``very weakly'' in $V_2(\U)$ (see Appendix for the proof).

The space $W^{1,q}_x(\U)$ ($1\leq q <\infty$)
denotes the Banach space consisting of
functions $u\in L^q(\U)$ with weak derivatives
$D_\alpha u \in L^q(\U)$ ($\alpha=1,\ldots,n$) with the norm
\begin{equation*}  
\norm{u}_{W^{1,q}_x(\U)}=\norm{u}_{L^q(\U)}+\norm{Du}_{L^q(\U)}.
\end{equation*}
We write $u \in L^\infty_c(\U)$ if $u \in L^\infty(\U)$ has a compact
support in $\U$.

\subsection{Strongly parabolic systems}
Throughout this article, the summation convention over repeated indices
are assumed.
Let $\cL=\partial_t-L$ be a second order parabolic operator of divergence type
acting on vector valued functions
$\vec{u}=(u^1,\ldots,u^N)^T$ ($N\ge 1$) defined on
an open subset of $\bR^{n+1}$ in the following way:
\begin{equation}
\label{eqn:P-01}
\cL\vec{u}=\vec{u}_t-L\vec{u}
:=\vec{u}_t-D_\alpha (\vec{A}^{\alpha\beta}\,D_\beta \vec{u}),
\end{equation}
where $\vec{A}^{\alpha\beta}$ ($\alpha,\beta=1,\ldots, n$)
are $N$ by $N$ matrix valued functions
with components $(A^{\alpha\beta}_{ij})_{i,j=1}^N$ defined on $\bR^{n+1}$
satisfying the strong parabolicity condition, i.e.,
there is a number $\lambda>0$ such that
\begin{equation}
\label{eqn:P-02}
A^{\alpha\beta}_{ij}(X)\xi^j_\beta\xi^i_\alpha
\ge \lambda \abs{\vec{\xi}}^2
:=\lambda\sum_{i=1}^N\sum_{\alpha=1}^n|\xi^i_\alpha|^2
\quad\forall X\in\bR^{n+1}.
\end{equation}
We also assume that $A^{\alpha\beta}_{ij}$ are bounded, i.e.,
there is a number $\Lambda>0$ such that
\begin{equation}
\label{eqn:P-03}
\sum_{i,j=1}^N\sum_{\alpha,\beta=1}^n
|A^{\alpha\beta}_{ij}(X)|^2\le \Lambda^2 \quad\forall X\in\bR^{n+1}.
\end{equation}
If we write \eqref{eqn:P-01} component-wise, then we have
\begin{equation*}
(\cL \vec{u})^i= u^i_t-D_\alpha (A^{\alpha\beta}_{ij} D_\beta u^j)
\quad \forall i=1,\ldots,N.
\end{equation*}
The transpose operator $\cLt$ of $\cL$ is defined by
\begin{equation*}    
\cLt\vec{u}=-\vec{u}_t-{}^t\!L\vec{u}
:=-\vec{u}_t-D_\alpha ({}^t\!\vec{A}^{\alpha\beta} D_\beta \vec{u}),
\end{equation*}
where ${}^t\!\vec{A}^{\alpha\beta}=(\vec{A}^{\beta\alpha})^T$
(i.e., ${}^t\!A^{\alpha\beta}_{ij}=A^{\beta\alpha}_{ji}$).
Note that the coefficients ${}^t\!A^{\alpha\beta}_{ij}$ also satisfy
\eqref{eqn:P-02} and \eqref{eqn:P-03}
with the same constants $\lambda, \Lambda$.

\subsection{Weak solutions}
For $\vec f, \vec g_\alpha \in L^1_{loc}(\U)^N$ ($\alpha=1,\ldots,n$),
we say that $\vec{u}$ is a weak solution of
$\cL \vec{u}=\vec{f}+D_\alpha\vec{g}_\alpha$
in $\U$ if $\vec{u}\in V_2(\U)^N$ and satisfies
\begin{equation}  \label{eqn:E-71}
-\int_{\U} u^i\phi^i_t+
\int_{\U} A^{\alpha\beta}_{ij}D_\beta u^j D_\alpha\phi^i=
\int_{\U} f^i \phi^i- \int_{\U} g^i_\alpha D_\alpha\phi^i
\quad\forall \vec{\phi}\in C^\infty_c(\U)^N.
\end{equation}
Similarly, we say that $\vec{u}$ is a weak solution of
$\cLt \vec{u}=\vec{f}+D_\alpha\vec{g}_\alpha$
in $\U$ if $\vec{u}\in V_2(\U)^N$ and satisfies
\begin{equation}  \label{eqn:E-71b}
\int_{\U} u^i\phi^i_t+
\int_{\U}{}^t\!A^{\alpha\beta}_{ij}D_\beta u^j D_\alpha\phi^i=
\int_{\U} f^i \phi^i- \int_{\U} g^i_\alpha D_\alpha\phi^i
\quad\forall \vec{\phi}\in C^\infty_c(\U)^N.
\end{equation}
By a weak solution in $\rV^{1,0}_2((a,b)\times\Omega)^N$
($-\infty<a<b<\infty$) of the problem
\begin{equation} \label{cauchy}
\left\{\begin{array}{l l}
\cL \vec{u}=\vec{f}\\
\vec{u}(a,\cdot)= \vec{g},\end{array}\right.
\end{equation}
we mean a function $\vec{u}(t,x)$ that belongs
to $\rV^{1,0}_2((a,b)\times\Omega)^N$ and
satisfying for all $t_1$ in $[a,b]$ the identity
\begin{align}
\nonumber
\int_\Omega u^i(t_1,\cdot)\phi^i(t_1,\cdot)
&-\int_a^{t_1}\!\!\!\int_\Omega u^i\phi^i_t
+\int_a^{t_1}\!\!\!\int_\Omega A^{\alpha\beta}_{ij}D_\beta u^j D_\alpha\phi^i\\
\label{eq:c01}
&\qquad=\int_a^{t_1}\!\!\!\int_\Omega f^i\phi^i+\int_\Omega g^i \phi^i(a,\cdot)
\quad\forall \vec{\phi}\in C^\infty_{c,p}((a,b)\times\Omega)^N.
\end{align}
Similarly, by a weak solution in $\rV^{1,0}_2((a,b)\times\Omega)^N$
of the (backward) problem
\begin{equation} \label{cauchy2}
\left\{\begin{array}{l l}
\cLt \vec{u}=\vec{f}\\
\vec{u}(b,\cdot)= \vec{g},\end{array}\right.
\end{equation}
we mean a function $\vec{u}(t,x)$ that belongs to
$\rV^{1,0}_2((a,b)\times\Omega)^N$ and satisfies for all $t_1$ in $[a,b]$
the identity
\begin{align}
\nonumber
\int_\Omega u^i(t_1,\cdot)\phi^i(t_1,\cdot)
&+\int^b_{t_1}\!\!\int_\Omega u^i\phi^i_t
+\int^b_{t_1}\!\!\int_\Omega {}^t\!A^{\alpha\beta}_{ij}D_\beta u^j
D_\alpha\phi^i\\
\label{eq:c02}
&\qquad=\int^b_{t_1}\!\!\int_\Omega f^i\phi^i+\int_\Omega g^i \phi^i(b,\cdot)
\quad\forall \vec{\phi}\in C^\infty_{c,p}((a,b)\times\Omega)^N.
\end{align}
We say that $\vec u$ is a weak solution in
$\rV^{1,0}_2((a,\infty)\times\Omega)^N$ of the problem \eqref{cauchy}
if $\vec u$ is a weak solution in $\rV^{1,0}_2((a,b)\times\Omega)$
of the problem \eqref{cauchy} for all $b>a$ and
$\tri{\vec u}_{(a,\infty)\times\Omega}<\infty$.
Similarly, we say that $\vec u$ is a weak solution in
$\rV^{1,0}_2((-\infty,b)\times\Omega)^N$ of the problem \eqref{cauchy2}
if $\vec u$ is a weak solution in $\rV^{1,0}_2((a,b)\times\Omega)$
of the problem \eqref{cauchy2} for all $a<b$ and
$\tri{\vec u}_{(-\infty,b)\times\Omega}<\infty$.

\begin{lemma} \label{lem13.1}
Assume that $\Omega$ is an open connected set in $\bR^n$,
$\vec{f}\in L^\infty_c(\bR\times\Omega)^N$, and $\vec{g}\in L^2(\Omega)^N$.
Then, 
there exists a unique weak solution in $\rV^{1,0}_2((a,\infty)\times\Omega)^N$
(resp. $\rV^{1,0}_2((-\infty,b)\times\Omega)^N$)
of the problem \eqref{cauchy} (resp. \eqref{cauchy2}).
\end{lemma}
\begin{proof}
See e.g., \cite[\S III.4]{LSU} for the existence and 
\cite[\S III.3]{LSU} for the uniqueness.
We point out that the proof does not require the boundedness of $\Omega$.
\end{proof}

\subsection{Property (PH)} \label{sec:PH}
We say that the operator $\cL$ satisfies the property (PH) if
there exist $\mu_0\in (0,1]$, $R_c\in (0,\infty]$, and $C_0>0$
such that all weak solutions $\vec{u}$
of $\cL\vec u=0$ in $Q_R^-=Q_R^-(X_0)$ with $R<R_c$ satisfy
\begin{equation} \label{eq3.43}
\int_{Q_\rho^-}\abs{D\vec u}^2\leq
C_0\left(\frac{\rho}{r}\right)^{n+2\mu_0}\int_{Q_r^-}\abs{D\vec{u}}^2
\quad \forall 0<\rho<r\leq R.
\end{equation}
Similarly, we say that the operator $\cLt$ satisfies
the property (PH) if all weak solutions $\vec u$ of $\cLt \vec u=0$
in $Q_R^+=Q_R^+(X_0)$ with $R<R_c$ satisfy
\begin{equation} \label{eq3.43b}
\int_{Q_\rho^+}\abs{D\vec{u}}^2\leq
C_0\left(\frac{\rho}{r}\right)^{n+2\mu_0}\int_{Q_r^+}\abs{D\vec{u}}^2
\quad \forall 0<\rho<r\leq R.
\end{equation}

Now, we present examples of strongly parabolic systems
satisfying the property (PH). Lemma~\ref{lem1} states that
\textit{almost diagonal} parabolic systems satisfy the property (PH)
with $R_c=\infty$ and Lemma~\ref{lem:V-03} shows that parabolic
systems with coefficients which belong to $\VMO_x$ (see below for
the definition) satisfy the property (PH) for some $R_c<\infty$.
In \cite{Kim}, it is shown that if $n=2$ and the coefficients of $\cL$ are
independent of $t$, then $\cL$ and $\cLt$ satisfy the property (PH)
with $R_c=\infty$.

\begin{lemma} \label{lem1}
Let $(a^{\alpha\beta}(X))_{\alpha,\beta=1}^n$ be coefficients satisfying
the following conditions:
There are constants $\lambda_0, \Lambda_0>0$ such that for all $X\in\bR^{n+1}$
\begin{equation*}
a^{\alpha\beta}(X)\xi_\beta\xi_\alpha\ge
\lambda_0\abs{\xi}^2 \quad\forall\xi\in\bR^n;\quad
\sum_{\alpha,\beta=1}^n\abs{a^{\alpha\beta}(X)}^2\le \Lambda_0^2.
\end{equation*}
Then, there exists $\epsilon_0=\epsilon_0(n,\lambda_0,\Lambda_0)>0$
such that if
\begin{equation*}
\epsilon^2(X):=
\sum_{i,j=1}^N\sum_{\alpha,\beta=1}^n
\abs{A^{\alpha\beta}_{ij}(X)-a^{\alpha\beta}(X)\delta_{ij}}^2< \epsilon_0^2
\quad\forall X\in\bR^{n+1},
\end{equation*}
where $\delta_{ij}$ is the Kronecker delta symbol,
then the operator $\cL$ associated with the coefficients
$A^{\alpha\beta}_{ij}$ and its transpose $\cLt$ satisfy
the property (PH) with $\mu_0=\mu_0(n,\lambda_0,\Lambda_0)$,
$C_0=C_0(n,N,\lambda_0,\Lambda_0)$, and $R_c=\infty$.
\end{lemma}
\begin{proof}
See e.g., \cite[Proposition 2.1]{HofKim2}.
\end{proof}

For a measurable function $f=f(X)=f(t,x)$ defined on $\bR^{n+1}$,
we set
\begin{equation*}
\omega_\delta(f):=\sup_{X=(t,x)\in\bR^{n+1}}\sup_{r\le \delta}
\frac{1}{\abs{Q_r(X)}}
\int_{t-r^2}^{t+r^2}\!\!\!\int_{B_r(x)} \abs{f(y,s)-\bar{f}_{x,r}(s)}\,dy\,ds
\quad\forall \delta>0,
\end{equation*}
where $\bar{f}_{x,r}(s)=\dashint_{B_r(x)}f(s,\cdot)$.
We say that $f$ belongs to $\VMO_x$
if $\lim_{\delta\to 0} \omega_\delta(f)=0$.
Note that $\VMO_x$ is a strictly larger class than the classical $\VMO$ space.
In particular, $\VMO_x$ contains all functions uniformly continuous in $x$
and measurable in $t$;
see \cite{Krylov}.

\begin{lemma}\label{lem:V-03}
Let the coefficients of the operator $\cL$ in \eqref{eqn:P-01}
satisfy the conditions \eqref{eqn:P-02} and \eqref{eqn:P-03}. If the
coefficients belong to $\VMO_x$, then the operators $\cL$ and $\cLt$
satisfy the property (PH).
\end{lemma}
\begin{proof}
Let $\vec{u}$ be a weak solution of $\cL \vec{u}=0$
in $Q_R^-(X_0)$ with $R<R_c$, where $R_c$ is to be chosen later.
First, note that as a weak solution of a linear strongly parabolic system,
$\vec u$ satisfies the following improved integrability estimates
for some $p=p(n,\lambda,\Lambda)>2$ (see \cite[Theorem 2.1]{GS}):
\begin{equation}\label{eq5.3}
\left(\dashint_{Q_r^-(X_0)}\abs{D\vec{u}}^p\right)^{2/p} \le
C\,\dashint_{Q_{2r}^-(X_0)}\abs{D\vec{u}}^2
\quad \forall r<R/2.
\end{equation}
For a given $r< R/2$, denote
$\bar{\vec A}{}^{\alpha\beta}_{x_0,r}=\bar{\vec A}{}^{\alpha\beta}_{x_0,r}(t)
=\dashint_{B_r(x_0)}{\vec A}^{\alpha\beta}(t,\cdot)$.
We decompose $\vec{u}=\vec{v}+\vec{w}$, where
$\vec{w}$ is the weak solution in $V^{1,0}_2(Q_r^-(X_0))^N$ satisfying
\begin{equation}  \label{eqn:A-10}
\vec{w}_t-D_\alpha\big(\bar{\vec A}{}^{\alpha\beta}_{x_0,r}
D_\beta\vec{w}\big)=
D_\alpha\big((\vec{A}^{\alpha\beta}-\bar{\vec A}{}^{\alpha\beta}_{x_0,r})
D_\beta\vec{u}\big) \quad\text{in } Q_r^-(X_0)
\end{equation}
with zero boundary condition on $\partial_p Q_r^-(X_0)$. By using
$\vec{w}$ itself as a test function in \eqref{eqn:A-10} and then
using H\"older's inequality, \eqref{eqn:P-03}, and \eqref{eq5.3},
we derive
\begin{align}
\nonumber
\int_{Q_r^-(X_0)}\abs{D\vec{w}}^2 &\leq
C\left(\int_{Q_r(X_0)}\abs{\vec{A}-\bar{\vec A}_{x_0,r}}^q\right)^{2/q}
\left(\int_{Q_r^-(X_0)}\abs{D\vec{u}}^p\right)^{2/p}\\
\label{eq5.4}
&\leq C\Lambda^{2(q-1)/q} \omega_r(\vec A)^{2/q}
\int_{Q_{2r}^-(X_0)}\abs{D\vec{u}}^2,\quad \text{where }q=2p/(p-2).
\end{align}
Observe that $\vec{v}=\vec{u}-\vec{w}$ is a weak solution of
\begin{equation*}
\vec{v}_t-D_\alpha\left(\bar{\vec A}{}^{\alpha\beta}_{x_0,r}(t)
D_\beta\vec{v}\right)=0
\quad\text{in }Q_r^-(X_0).
\end{equation*}
Since $\vec v$ is a weak solution of a strongly parabolic system with
coefficients independent of the spatial variables, $\vec v$ possesses
the following interior estimate:
\begin{equation}
\label{eqn:A-11}
\int_{Q_\rho^-(X_0)}\abs{D \vec{v}}^2\leq
C\left(\frac \rho r\right)^{n+2}
\int_{Q_r^-(X_0)}\abs{D \vec{v}}^2 \quad \forall \rho<r,
\end{equation}
for some constants $C=C(n,N,\lambda,\Lambda)$ (see Appendix).
Then, by combining \eqref{eq5.4} and \eqref{eqn:A-11}, we see
that for all $\rho<r<R/2$, we have
\begin{equation*}
\int_{Q_\rho^-(X_0)}\abs{D\vec{u}}^2 \leq
C\left(\frac{\rho}{r}\right)^{n+2}
\int_{Q_{2r}^-(X_0)}\abs{D\vec{u}}^2 +C \omega_r(\vec A)^{2/q}
\int_{Q_{2r}^-(X_0)}\abs{D\vec{u}}^2.
\end{equation*}
If we choose $R_c$ sufficiently small so that $C\omega_{R_c}(\vec A)^{2/q}$
is small, then by a well-known iteration argument we
obtain \eqref{eq3.43} (see e.g., \cite[Lemma~2.1, p. 86]{Giaq83}).
Therefore, we have proved that $\cL$ satisfies (PH).
The proof that $\cLt$ also satisfies (PH) is similar and left to the reader.
\end{proof}

\subsection{Some preliminary lemmas}
\begin{lemma} \label{lem3.1}
There exists a constant $C=C(n,N,\Lambda)>0$ such that if
$\vec{u}$ is a weak solution of $\cL \vec{u}= \vec{f}$ in
$Q_R^-=Q_R^-(X_0)$, then
\begin{equation*}    
\int_{Q_R^-} \abs{\vec{u}-\bar{\vec{u}}_R}^2
\leq C R^2\int_{Q_R^-}\abs{D\vec{u}}^2
+C R^{2-n}\norm{\vec{f}}_{L^1(Q_R^-)}^2,
\end{equation*}
where $\bar{\vec{u}}_R=\dashint_{Q_R^-}\vec{u}$.
Similarly, if $\vec{u}$ is a weak solution of $\cLt \vec{u}= \vec{f}$ in
$Q_R^+$, then
\begin{equation*}    
\int_{Q_R^+} \abs{\vec{u}-\bar{\vec{u}}_R}^2
\leq C R^2\int_{Q_R^+}\abs{D\vec{u}}^2
+C R^{2-n}\norm{\vec{f}}_{L^1(Q_R^+)}^2,
\end{equation*}
where $\bar{\vec{u}}_R=\dashint_{Q_R^+}\vec{u}$.
\end{lemma}
\begin{proof}
See e.g., \cite[Lemma 3]{Struwe}.
\end{proof}

\begin{lemma} \label{lem4}
Let $\vec{u}\in L^2(Q_{2R}^-(X_0))^N$ and suppose that there exist  positive
constants $\mu\in (0,1]$ and $N_0$ such that
\begin{equation*}
\int_{Q^-_r(X)}\abs{\vec{u}-\bar{\vec u}_{X,r}}^2\leq N_0^2r^{n+2+2\mu}
\quad \forall X\in Q_R^-(X_0) \quad \forall r\in (0,R),
\end{equation*}
where $\bar{\vec{u}}_{X,r}=\dashint_{Q_r^-(X)}\vec{u}$.
Then, $\vec{u}$ is H\"older continuous in $Q_R^-(X_0)$ and
\begin{equation*}
[\vec{u}]_{C^{\mu,\mu/2}(Q_R^-(X_0))}\leq C(n,N,\mu)N_0.
\end{equation*}
The same is true with $Q^+$ in place of $Q^-$ everywhere.
\end{lemma}
\begin{proof}
See e.g., \cite[Lemma 4.3]{Lieberman}.
\end{proof}

\begin{lemma} \label{lem3}
Assume that the operator $\cL$ satisfies the property (PH).
Then, all weak solutions $\vec{u}$ of $\cL\vec{u}=0$ in
$Q_R^-=Q_R^-(X_0)$ with $R<R_c$ satisfy
\begin{equation} \label{eq4.03}
\norm{\vec{u}}_{L^\infty(Q^-_{R/4}(X_0))}
\leq C \left(\dashint_{Q^-_R(X_0)}\abs{\vec{u}}^2\right)^{1/2},
\end{equation}
where $C=C(n,N,\lambda,\Lambda,\mu_0,C_0)>0$.
Moreover, all weak solutions $\vec{u}$ of $\cL\vec{u}=0$ in $Q_R^-=Q_R^-(X_0)$
with $R<R_c$ satisfy
\begin{equation}
\label{eqn:P-11}
\norm{\vec{u}}_{L^\infty(Q_r^-)}\le \frac{C_p}{(R-r)^{(n+2)/p}}
\norm{\vec{u}}_{L^p(Q_R^-)} \quad \forall r \in (0,R)
\quad\forall p>0,
\end{equation}
where $C_p=C_p(n,N,\lambda,\Lambda,\mu_0,C_0)>0$.
A similar statement is also true for $\cLt$ with
$Q^-$ replaced by $Q^+$ everywhere.
\end{lemma}
\begin{proof}
Let $\vec{u}$ be a weak solution of $\cL \vec{u}=0$ in $Q_R^-(X_0)$,
where $R<R_c$.
For any $X\in Q_{R/2}^-(X_0)$ and $r<R/4$, it follows from
Lemma~\ref{lem3.1}, property (PH), and the energy inequality
(see e.g., \cite[\S III.2]{LSU}) that
\begin{align} 
\nonumber
\int_{Q_r^-(X)} \abs{\vec{u}-\bar{\vec{u}}_{X,r}}^2
&\le C r^2 \int_{Q_r^-(X)} \abs{D\vec{u}}^2
\le C r^2(r/R)^{n+2\mu_0}
\int_{Q_{R/4}^-(X)}\abs{D\vec{u}}^2\\
\nonumber
&\le C (r/R)^{n+2+2\mu_0} \int_{Q_{R/2}^-(X)}\abs{\vec{u}}^2\\
\label{eq2.18}
&\le C r^{n+2+2\mu_0} R^{-2\mu_0} \dashint_{Q_R^-(X_0)}\abs{\vec{u}}^2.
\end{align}
Note that \eqref{eq2.18} also holds for any $X\in Q_{R/2}^-(X_0)$ and
any $r \in [R/4,R/2)$, because
\begin{equation*}
\int_{Q_r^-(X)} \abs{\vec{u}-\bar{\vec{u}}_{X,r}}^2
\le \int_{Q_r^-(X)} \abs{\vec{u}}^2
\le \int_{Q_R^-(X_0)} \abs{\vec{u}}^2.
\end{equation*}
Therefore, by Lemma~\ref{lem4}, we obtain
\begin{equation}
\label{eqn:P-12}
[\vec{u}]_{C^{\mu_0,\mu_0/2}(Q_{R/2}^-(X_0))}^2
\leq C R^{-2\mu_0} \dashint_{Q_R^-(X_0)}\abs{\vec{u}}^2.
\end{equation}
Then, \eqref{eq4.03} is an easy consequence of \eqref{eqn:P-12}
and a well known averaging argument (see e.g., \cite{HofKim2}).
We point out that it actually follows from \eqref{eq2.18}
\begin{equation}
\label{eqP.24}
[\vec{u}]_{C^{\mu_0,\mu_0/2}(Q_{R/2}^-(X_0))}^2
\leq C R^{2-2\mu_0} \dashint_{Q_R^-(X_0)}\abs{D\vec{u}}^2.
\end{equation}
For the proof that \eqref{eq4.03} implies \eqref{eqn:P-11},
we refer to \cite[pp. 80--82]{Giaq93}.
\end{proof}

\subsection{Main results}
We now state our main theorems.
\begin{theorem} \label{thm1}
Let $\U=\bR\times\Omega$, where $\Omega$ is an open connected set in $\bR^n$.
Denote $d_X=\dist(X,\partial\U)=\dist(x,\partial\Omega)$ for $X=(t,x)\in \U$;
we set $d_X=\infty$ if $\Omega=\bR^n$.
Assume that operators $\cL$ and $\cLt$ satisfy the property $\PH$.
Then, there exists a unique Green's matrix
$\vec{\Gamma}(X,Y)=\vec{\Gamma}(t,x,s,y)$ on $\U\times\U$
which is continuous in $\set{(X,Y)\in\U\times\U:X\neq Y}$,
satisfies $\vec{\Gamma}(t,x,s,y)\equiv 0$ for $t<s$,
and has the property that $\vec\Gamma(X,\cdot)$ is locally integrable in $\U$
for all $X\in\U$ and that for all
$\vec{f}\in C^\infty_c(\U)^N$, the function $\vec{u}$ given by
\begin{equation}  \label{eqn:E-70}
\vec{u}(X):=\int_{\U} \vec{\Gamma}(X,Y)\vec{f}(Y)\,dY
\end{equation}
belongs to $\rV^{1,0}_2(\U)^N$ and satisfies $\cL\vec{u}=\vec{f}$ in the sense
of \eqref{eqn:E-71}.
Moreover, $\vec\Gamma$ satisfies
\begin{equation}  \label{eq5.22}
\int_{\U}\left(-\Gamma_{ik}(\cdot,Y)\phi^i_t+
A^{\alpha\beta}_{ij}D_\beta \Gamma_{jk}(\cdot,Y)D_\alpha \phi^i\right)
= \phi^k(Y)
\quad \forall \vec{\phi}\in C^\infty_c(\U)^N
\end{equation}
and for all $\eta\in C^\infty_c(\U)$ satisfying $\eta\equiv 1$ on $Q_r(Y)$
for some $r<d_Y$, we have
\begin{equation}  \label{eq5.21}
(1-\eta)\vec\Gamma(\cdot,Y)\in \rV^{1,0}_2(\U)^{N\times N}.
\end{equation}
Furthermore, for all $\vec{g}\in L^2(\Omega)^N$,
the function $\vec{u}(t,x)$ given by
\begin{equation}  \label{eq5.24}
\vec{u}(t,x):=\int_{\Omega}\vec{\Gamma}(t,x,s,y)\vec{g}(y)\,dy
\qquad \forall x\in\Omega \quad \forall t>s
\end{equation}
is the unique weak solution in $\rV^{1,0}_2((s,\infty)\times\Omega)^N$
of the Cauchy problem
\begin{equation} \label{eq5.88}
\left\{\begin{array}{l l}
\cL \vec{u}=0\\
\vec{u}(s,\cdot)= \vec{g},\end{array}\right.
\end{equation}
and if $\vec g$ is continuous at $x_0\in\Omega$ in addition, then
\begin{equation} \label{eq5.23}
\lim_{\substack{(t,x)\to (s,x_0)\\x\in\Omega,\,t>s}}
\int_{\Omega}\vec{\Gamma}(t,x,s,y)\vec{g}(y)\,dy =\vec{g}(x_0).
\end{equation}
Denote $\bar{d}_X:=\min(d_X, R_c)$ for $X\in\U$.
Then, $\vec{\Gamma}$ satisfies the following estimates:
\begin{gather} \label{eq5.15}
\|\vec{\Gamma}(\cdot,Y)\|_{L^{2+4/n}(\U\setminus \overline Q_r(Y))}+
\tri{\vec{\Gamma}(\cdot,Y)}_{\U\setminus \overline Q_r(Y)}
\leq Cr^{-n/2} \quad \forall r<\bar{d}_Y,\\
\label{eq5.15b}
\|\vec{\Gamma}(X,\cdot)\|_{L^{2+4/n}(\U\setminus \overline Q_r(X))}+
\tri{\vec{\Gamma}(X,\cdot)}_{\U\setminus \overline Q_r(X)}
\leq Cr^{-n/2} \quad \forall r<\bar{d}_X,
\end{gather}
\begin{gather} \label{eq5.18}
\|\vec{\Gamma}(\cdot,Y)\|_{L^p(Q_r(Y))} \leq C_p r^{-n+(n+2)/p}
\quad \forall r<\bar{d}_Y \quad\forall p\in [1,\tfrac{n+2}{n}),\\
\|\vec{\Gamma}(X,\cdot)\|_{L^p(Q_r(X))}\leq C_p r^{-n+(n+2)/p}
\quad \forall r<\bar{d}_X \quad\forall p\in [1,\tfrac{n+2}{n}),
\end{gather}
\begin{gather}
\abs{\set{X\in \U:\abs{\vec{\Gamma}(X,Y)}>\tau}} \leq C \tau^{-(n+2)/n}
\quad \forall \tau>(\bar{d}_Y/2)^{-n} \quad \forall Y\in\U,\\
\abs{\set{Y\in \U:\abs{\vec{\Gamma}(X,Y)}>\tau}} \leq C \tau^{-(n+2)/n}
\quad \forall \tau>(\bar{d}_X/2)^{-n} \quad \forall X\in\U,
\end{gather}
\begin{gather} \label{eq5.19}
\|D\vec{\Gamma}(\cdot,Y)\|_{L^p(Q_r(Y))}\leq C_p r^{-n-1+(n+2)/p}
\quad \forall r<\bar{d}_Y \quad\forall p\in [1,\tfrac{n+2}{n+1}),\\
\|D\vec{\Gamma}(X,\cdot)\|_{L^p(Q_r(X))}\leq C_p r^{-n-1+(n+2)/p}
\quad \forall r<\bar{d}_X \quad\forall p\in [1,\tfrac{n+2}{n+1}),
\end{gather}
\begin{gather}
\abs{\set{X\in \U:\abs{D_x\vec{\Gamma}(X,Y)}>\tau}}
\leq C \tau^{-\frac{n+2}{n+1}}
\quad \forall \tau>(\bar{d}_Y/2)^{-n} \quad \forall Y\in\U,\\
\label{eq5.20}
\abs{\set{Y\in \U:\abs{D_y\vec{\Gamma}(X,Y)}>\tau}}
\leq C \tau^{-\frac{n+2}{n+1}}
\quad \forall \tau>(\bar{d}_X/2)^{-n} \quad \forall X\in\U,
\end{gather}
\begin{equation} \label{eq365}
\abs{\vec\Gamma(X,Y)}\leq C |X-Y|_p^{-n}\quad
\text{if}\quad 0<|X-Y|_p<\tfrac{1}{2}\max(\bar{d}_X,\bar{d}_Y),
\end{equation}
where $C=C(n,N,\lambda,\Lambda,\mu_0,C_0)>0$ and
$C_p=C_p(n,N,\lambda,\Lambda,\mu_0,C_0,p)>0$.
\end{theorem}

\begin{corollary} \label{cor1}
If $N=1$ or if $N>1$ and the coefficients of $\cL$ satisfies the assumption
of Lemma~\ref{lem1}, then there exists a unique Green's function
in a cylindrical domain satisfying the properties in
Theorem~\ref{thm1} with $R_c=\infty$.
\end{corollary}
\begin{proof}
See Lemma~\ref{lem1}.
\end{proof}

\begin{corollary} \label{cor2}
If $n=2$ and the coefficients of $\cL$ are independent of $t$, then
there exists a unique Green's matrix
in a cylindrical domain satisfying the properties in
Theorem~\ref{thm1} with $R_c=\infty$.
\end{corollary}
\begin{proof}
See \cite{Kim} for the proof that $\cL$ and $\cLt$ satisfy the property $\PH$.
\end{proof}

\begin{corollary} \label{cor3}
If $N>1$ and the coefficients of $\cL$ belong to $\VMO_x$
(e.g. the coefficients are uniformly continuous in $x$ and
measurable in $t$ or belong to $\VMO$),
then there exists a unique Green's matrix
in a cylindrical domain satisfying the properties in
Theorem~\ref{thm1} with $R_c>0$.
\begin{proof}
See Lemma~\ref{lem:V-03}.
\end{proof}

\end{corollary}

\begin{theorem} \label{thm2}
Assume that operators $\cL$ and $\cLt$ satisfy the property $\PH$
and let $\vec\Gamma(t,x,s,y)$ be the fundamental matrix for
$\cL$ in $\U=\bR^{n+1}$
as given in Theorem~\ref{thm1}.
Then, we have for all $t>s$ and $x,y\in\bR^n$
\begin{equation} \label{eq5.27}
\abs{\vec{\Gamma}(t,x,s,y)}_{op}\leq
C(t-s)^{-n/2}\exp\{-\kappa|x-y|^2/(t-s)+\gamma(t-s)/R_c^2\},
\end{equation}
where $C=C(n,N,\lambda,\Lambda,\mu_0,C_0), \kappa=\kappa(\lambda,\Lambda),
\gamma=\gamma(n,N,\lambda,\Lambda,\mu_0,C_0)>0$,
and $\abs{\,\cdot\,}_{op}$ denotes the operator norm.
In particular, when $R_c=\infty$, we have the following usual Gaussian bound
for all $t>s$ and $x,y\in\bR^n$:
\begin{equation} \label{eq5.27b}
\abs{\vec{\Gamma}(t,x,s,y)}_{op}\leq
C(t-s)^{-n/2}\exp\{-\kappa|x-y|^2/(t-s)\}.
\end{equation}
Moreover, for $t>s$ the following identities hold:
\begin{gather} 
\label{eq5.14z}
\vec\Gamma(t,x,s,y)=\int_{\bR^n} \vec\Gamma(t,x,r,z)\vec\Gamma(r,z,s,y)\,dz
\quad\forall x,y\in\bR^n \quad\forall r\in (s,t),\\
\label{eq5.99}
\int_{\bR^n}\vec{\Gamma}(t,x,s,y)\,dy=\vec{I}
\quad\forall x\in\bR^n,
\end{gather}
where $\vec{I}$ is the $N$ by $N$ identity matrix.
\end{theorem}
\begin{remark}
In fact, \eqref{eq5.14z} also holds for $\Omega\neq \bR^n$;
see Section~\ref{sec:5-2}.
\end{remark}

\mysection{Proof of Theorem \ref{thm1}: when $\Omega=\bR^n$ and $R_c=\infty$}
\label{sec3}
In this section we shall prove Theorem~\ref{thm1} under
the assumption that $\Omega=\bR^n$ and that the operators $\cL$ and $\cLt$
satisfy the property (PH) with $R_c=\infty$.
Consequently, we have $\U=\bR^{n+1}$ and $d_X=\bar{d}_X=\infty$
for all $X\in \bR^{n+1}$.
Throughout this section, we employ the letter $C$ to denote a constant
depending on $n$, $N$, $\lambda$, $\Lambda$, $\mu_0$, $C_0$, and sometimes 
on an exponent $p$ characterizing Lebesgue classes.
\subsection{Averaged fundamental matrix}  \label{sec0301}
Our approach here is an adaptation of that in Hofmann-Kim \cite{HofKim1}, which
in turn is partly based on the method by Gr\"uter and Widman \cite{GW}.
Let $Y=(s,y)\in \bR^{n+1}$ and $1\le k \le N$ be fixed.
For each $\rho>0$, fix $s_0\in (-\infty,s-\rho^2)$.
We consider the problem
\begin{equation} \label{eq1.39}
\left\{\begin{array}{l l}
\cL \vec{u}=\frac{1}{\abs{Q^{-}_\rho}}1_{Q^-_\rho(Y)} \vec{e}_k\\
\vec{u}(s_0,\cdot)= 0,\end{array}\right.
\end{equation}
where $\vec{e}_k$ is the $k$-th unit vector.
By Lemma~\ref{lem13.1}, we find that the problem \eqref{eq1.39}
has a unique weak solution $\vec{v}_\rho=\vec{v}_{\rho;Y,k}$
in $\rV^{1,0}_2((s_0,\infty)\times\bR^n)^N$.
Moreover, by the uniqueness, we find that $\vec{v}_\rho$ does not depend
on the particular choice of $s_0$ and we may extend $\vec{v}_\rho$ to
the entire $\bR^{n+1}$ by setting
\begin{equation} \label{eq11.54}
\vec{v}_\rho\equiv 0\quad\text{on}\quad(-\infty,s-\rho^2)\times \bR^n.
\end{equation}
Then $\vec{v}_\rho \in \rV^{1,0}_2(\bR^{n+1})^N$ and satisfies for
all $t_1>s$ the identity
\begin{equation} \label{eq2.24}
\int_{\bR^n}v_\rho^i \phi^i(t_1,\cdot)
-\int_{-\infty}^{t_1}\int_{\bR^n}v_\rho^i \phi^i_t
+\int_{-\infty}^{t_1}\int_{\bR^n}
A^{\alpha\beta}_{ij}D_\beta v_\rho^j D_\alpha \phi^i
= \dashint_{Q^-_\rho(Y)}\phi^k
\end{equation}
for all $\vec{\phi}\in C^\infty_{c,p}(\bR^{n+1})^N$.
We define the \textit{averaged fundamental matrix}
$\vec{\Gamma}^\rho(\cdot,Y)=(\Gamma^\rho_{jk}(\cdot,Y))_{j,k=1}^N$
for $\cL$ by
\begin{equation*}    
\Gamma^\rho_{jk}(\cdot,Y)=v_\rho^j=v^j_{\rho;Y,k}.
\end{equation*}

Next, for each $\vec{f}\in L^\infty_c(\bR^{n+1})^N$,
let us fix $t_0$ such that $\vec{f}\equiv 0$ on $[t_0,\infty)\times \bR^n$.
We consider the backward problem
\begin{equation} \label{eq2.05}
\left\{\begin{array}{l l}
\cLt \vec{u}=\vec{f}\\
\vec{u}(t_0,\cdot)= 0.\end{array}\right.
\end{equation}
Again, by Lemma~\ref{lem13.1} we obtain a unique weak solution $\vec{u}$
in $\rV^{1,0}_2((-\infty,t_0)\times\bR^n)^N$ of the problem \eqref{eq2.05} and
we may extend $\vec{u}$ to the entire $\bR^{n+1}$ by setting
$\vec{u}\equiv 0$ on $(t_0,\infty)\times\bR^n$.
Then, $\vec{u}\in \rV^{1,0}_2(\bR^{n+1})^N$ and satisfies for
all $t_1$ the identity
\begin{equation} \label{eq2.35}
\int_{\bR^n}u^i \phi^i(t_1,\cdot)
+\int_{t_1}^\infty\!\int_{\bR^{n}}u^i \phi^i_t
+\int_{t_1}^\infty\!\int_{\bR^n}
{}^t\!A^{\alpha\beta}_{ij}D_\beta u^j D_\alpha \phi^i
=\int_{t_1}^\infty\!\int_{\bR^n} f^i \phi^i
\end{equation}
for all $\vec{\phi}\in C^\infty_{c,p}(\bR^{n+1})$.

\begin{lemma} \label{lem3.1.1}
For $\vec{v}_\rho$ and $\vec{u}$ constructed above, we have
\begin{gather} \label{eq5.58}
\tri{\vec{v}_\rho}_{\bR^{n+1}}\leq C\abs{Q_\rho^-(Y)}^{-\frac{n}{2n+4}},\\
\label{eq2.00}
\tri{\vec{u}}_{\bR^{n+1}}
\leq C \norm{\vec{f}}_{L^{\frac{2n+4}{n+4}}(\bR^{n+1})},\\
\label{eq2.17}
\int_{\bR^{n+1}} \vec{v}_\rho\cdot\vec{f}
=\dashint_{Q^-_\rho(Y)} u^k.
\end{gather}
\end{lemma}
\begin{proof}
The energy inequality (see e.g., \cite[\S III.2]{LSU})
together with \eqref{eqn:2.2} yields
\eqref{eq5.58} and \eqref{eq2.00}.
The identity \eqref{eq2.17} follows from \eqref{eq2.24} and \eqref{eq2.35}
accompanied by standard approximation techniques
(see e.g., \cite[\S III.2]{LSU} or \cite[\S VI.1]{Lieberman} for details).
\end{proof}

\subsection{$L^\infty$ estimate for averaged fundamental matrix}
\label{sec3.2}
Let $\vec{u}\in\rV^{1,0}_2(\bR^{n+1})^N$ be constructed as above
with $\vec f\in L^\infty_c(\bR^{n+1})^N$.
Fix $X_0\in \bR^{n+1}$, $R>0$, $X\in Q^+_R(X_0)$, and $r\in (0,R]$.
We decompose $\vec{u}=\vec{u}_1+\vec{u}_2$, where $\vec{u}_2$
is the unique weak solution in $V^{1,0}_2(Q^{+}_r(X))$ of
$\cLt \vec{u}_2=\vec{f}$ in $Q^{+}_r(X)$
with zero boundary condition on $\partial_p Q^{+}_r(X)$.
Then, $\vec{u}_1=\vec{u}-\vec{u}_2$ satisfies
$\cLt \vec{u}_1=0$ in $Q^{+}_r(X)$, and thus, for $0<\delta<r$,
\begin{align}
\int_{Q^+_\delta(X)}\abs{D\vec{u}}^2&\leq
2\int_{Q^+_\delta(X)}\abs{D\vec{u}_1}^2
+2\int_{Q^+_\delta(X)}\abs{D\vec{u}_2}^2\nonumber\\
&\leq C(\delta/r)^{n+2\mu_0}\int_{Q^+_r(X)}\abs{D\vec{u}_1}^2
+2\int_{Q^+_r(X)}\abs{D\vec{u}_2}^2\nonumber\\
\label{eq7.32}
&\leq C(\delta/r)^{n+2\mu_0}\int_{Q^+_r(X)}\abs{D\vec{u}}^2
+C\int_{Q^+_r(X)}\abs{D\vec{u}_2}^2.
\end{align}
For a given $p>(n+2)/2$, choose $p_0\in ((n+2)/2,p)$ such that
\begin{equation*}    
\mu_1:=2-(n+2)/p_0<\mu_0.
\end{equation*}
As in \eqref{eq2.00}, we have
\begin{equation} \label{eq7.33}
\int_{Q^+_r(X)}\abs{D\vec{u}_2}^2\leq
C\norm{\vec f}^2_{L^{\frac{2n+4}{n+4}}(Q^+_r(X))}
\leq C r^{n+2\mu_1} \norm{\vec f}^2_{L^{p_0}(Q^+_r(X))}.
\end{equation}
Combining \eqref{eq7.32} with \eqref{eq7.33}, we get for all $\delta<r\le R$,
\begin{equation*}
\int_{Q^+_\delta(X)}\abs{D\vec{u}}^2\leq
C(\delta/r)^{n+2\mu_0}\int_{Q^+_r(X)}\abs{D\vec{u}}^2+
Cr^{n+2\mu_1}\norm{\vec f}^2_{L^{p_0}(Q^+_r(X))}.
\end{equation*}
Then, by a well known iteration argument
(see e.g., \cite[Lemma 2.1, p. 86]{Giaq83}),
\begin{equation} \label{eq3.30}
\int_{Q^+_r(X)}\abs{D\vec{u}}^2\leq
C(r/R)^{n+2\mu_1}\int_{Q^+_R(X)}\abs{D\vec{u}}^2+
Cr^{n+2\mu_1}\norm{\vec f}^2_{L^{p_0}(Q^+_R(X))}.
\end{equation}
By Lemma \ref{lem3.1}, \eqref{eq3.30}, and H\"older's inequality, we get
\begin{equation*}
\int_{Q_r^+(X)}\abs{\vec{u}-\bar{\vec{u}}_r}^2\leq
Cr^{2+n+2\mu_1}\left(R^{-n-2\mu_1}\norm{D\vec u}^2_{L^2(Q_R^+(X))}
+\norm{\vec f}^2_{L^{p_0}(Q_R^+(X))}\right).
\end{equation*}
Then, from Lemma \ref{lem4} and \eqref{eq2.00}, it follows that
\begin{align}
[\vec{u}]^2_{C^{\mu_1,\mu_1/2}(Q^+_R(X_0))}&\leq
C\left(R^{-n-2\mu_1}\norm{D\vec u}^2_{L^2(\bR^{n+1})}
+\norm{\vec f}^2_{L^{p_0}(\bR^{n+1})}\right)\nonumber \\
\label{eq10.191}
&\leq C\left(R^{-n-2\mu_1}\norm{\vec f}^2_{L^{\frac{2n+4}{n+4}}(\bR^{n+1})}
+\norm{\vec f}^2_{L^{p_0}(\bR^{n+1})}\right).
\end{align}
By H\"older's inequality, \eqref{eqn:2.2}, and \eqref{eq2.00},
\begin{equation} \label{eq10.28}
\norm{\vec u}^2_{L^2(Q^+_R(X_0))}\leq
C R^2\norm{\vec f}^2_{L^{\frac{2n+4}{n+4}}(\bR^{n+1})}.
\end{equation}
By \eqref{eq10.191}, \eqref{eq10.28}, and a standard averaging method
(see e.g., \cite{HofKim2}), we obtain
\begin{align}
\nonumber
\norm{\vec u}_{L^\infty(Q_{R/2}^+(X_0))}^2&\leq
CR^{2\mu_1}[\vec{u}]^2_{C^{\mu_1,\mu_1/2}(Q^+_R(X_0))}+CR^{-n-2}
\norm{\vec u}_{L^2(Q^+_R(X_0))}^2\\
\label{eq07zz}
&\leq C R^{-n}\norm{\vec f}^2_{L^{\frac{2n+4}{n+4}}(\bR^{n+1})}
+CR^{2\mu_1}\norm{\vec f}^2_{L^{p_0}(\bR^{n+1})}.
\end{align}
In the remaining part of this subsection, we shall assume that $\vec{f}$
is supported in $Q^+_R(X_0)$.
Recall that $p>(n+2)/2$.
We apply H\"older's inequality in \eqref{eq07zz} to get
\begin{equation} \label{eq2.19}
\norm{\vec u}_{L^\infty(Q_{R/2}^+(X_0))}
\leq C R^{2-(n+2)/p}\norm{\vec f}_{L^p(Q_R^+(X_0))}.
\end{equation}

If $Q^-_\rho(Y)\subset Q^+_{R/2}(X_0)$, then \eqref{eq2.17} together with
\eqref{eq2.19} yields
\begin{equation*}  
\abs{\int_{Q^+_R(X_0)}\vec{v}^\rho \cdot \vec{f}} \leq
\dashint_{Q^-_\rho(Y)}\abs{\vec u}\leq
CR^{2-(n+2)/p}\norm{\vec f}_{L^{p}(Q^+_R(X_0))}
\end{equation*}
for any $p>(n+2)/2$.
By duality, it follows that if $Q^-_\rho(Y)\subset Q^+_{R/2}(X_0)$, then
\begin{equation} \label{eq11.09}
\norm{\vec v_\rho}_{L^q(Q^+_R(X_0))}\leq CR^{-n+(n+2)/q}
\quad \forall q\in [1, (n+2)/n).
\end{equation}
Consequently, we obtain the following pointwise estimate for
the averaged fundamental matrices, which is the main result of this
subsection.

\begin{lemma} \label{lem3.2.5}
Let $X=(t,x)$, $Y=(s,y)$, and assume $X\neq Y$.
Then
\begin{equation} \label{eq11.16}
\abs{\vec{\Gamma}^\rho(X,Y)}\leq C \abs{X-Y}_p^{-n}
\quad \forall\rho\leq \abs{X-Y}_p/3.
\end{equation}
\end{lemma}
\begin{proof}
Denote $d=\abs{X-Y}_p$ and let $X_0=(s-4d^2,y)$, $r=d/3$, and $R=20r$.
It is easy to see that (recall $\rho\leq r$)
\begin{equation*}
Q^{-}_\rho(Y)\subset Q^{+}_{R/2}(X_0),\quad Q^{-}_r(X)\subset Q^{+}_{R}(X_0).
\end{equation*}
Moreover, by \eqref{eq2.24}, $\vec{v}_\rho=\vec{v}_{\rho;Y,k}$ is a weak
solution of
$\cL \vec u=0$ in $Q^-_r(X)$.
Therefore, by Lemma \ref{lem3} and \eqref{eq11.09}, we have
\begin{equation*}
\abs{\vec{v}_\rho(X)}\leq Cr^{-n-2}\norm{\vec{v}_\rho}_{L^1(Q_r^{-}(X))}
\leq Cr^{-n-2}\norm{\vec{v}_\rho}_{L^1(Q_R^{+}(X_0))}\leq Cr^{-n}.
\end{equation*}
The lemma is proved.
\end{proof}

\subsection{Construction of the fundamental matrix}	\label{sec0303}
In next two lemmas, we derive $L^p$ estimates uniform in $\rho>0$
for averaged fundamental matrices $\vec\Gamma^\rho(\cdot,Y)$
and their spatial derivatives $D\vec\Gamma^\rho(\cdot,Y)$.
\begin{lemma} \label{lem3.3.2}
For any $Y\in\bR^{n+1}$ and any $\rho>0$, we have
\begin{equation} \label{eq11.17}
\abs{\set{X\in \bR^{n+1}:\abs{D_x\vec{\Gamma}^\rho(X,Y)}>\tau}}
\leq C \tau^{-\frac{n+2}{n+1}} \quad \forall \tau>0,
\end{equation} i.e., $\abs{D\vec{\Gamma}^\rho(\cdot,Y)}$
has a weak-$L^{(n+2)/(n+1)}$ estimate in $\bR^{n+1}$.
Moreover, for any $R>0$,  the following uniform $L^p$ estimate
for $D\vec\Gamma^\rho(\cdot,Y)$ holds:
\begin{equation} \label{eq11.29}
\norm{D\vec{\Gamma}^\rho(\cdot,Y)}_{L^p(Q_R(Y))} \leq C_p R^{-n-1+(n+2)/p}
\quad\forall\rho>0 \quad \forall p\in [1,\tfrac{n+2}{n+1}).
\end{equation}
\end{lemma}
\begin{proof}
We claim that for all $R>0$, we have
\begin{equation} \label{eq11.299}
\tri{\vec{v}_\rho}_{\bR^{n+1}\setminus \overline Q_R(Y)}^2 \leq C R^{-n}
\quad \forall \rho>0.
\end{equation}
To prove the claim \eqref{eq11.299}, we only need to consider the case
$R>6\rho$.
Indeed, if $R\leq 6\rho$, then \eqref{eq5.58} yields
\begin{equation*}
\tri{\vec{v}_\rho}_{\bR^{n+1}\setminus \overline Q_R(Y)}^2
\leq \tri{\vec{v}_\rho}_{\bR^{n+1}}^2
\leq C\rho^{-n} \leq CR^{-n}.
\end{equation*}
Fix a cut-off function $\zeta\in C^\infty_c(Q_R(Y))$ such that
\begin{equation*}
\zeta\equiv 1\text{ on } Q_{R/2}(Y),\quad
0\leq \zeta \leq 1,\quad \abs{D\zeta}\leq C R^{-1},
\quad \abs{\zeta_t}\leq CR^{-2}.
\end{equation*}
By following a standard proof of the energy inequality
(see e.g., \cite[\S III.2]{LSU}), we derive from \eqref{eq1.39}
and \eqref{eq11.16} that
\begin{align} \nonumber
\sup_{t\in \bR}\int_{\bR^n}(1-\zeta)^2\abs{\vec{v}_\rho(t,\cdot)}^2
&+\int_{\bR^{n+1}}(1-\zeta)^2\abs{D\vec{v}_\rho}^2\\
\nonumber
&\leq C\int_{\bR^{n+1}}\left(\abs{D\zeta}^2+\abs{(1-\zeta)\zeta_t}\right)
\abs{\vec{v}_\rho}^2 \\
\label{eq5.59}
&\leq C R^{-2} \int_{\set{R/2<|X-Y|_p<R}}|X-Y|_p^{-2n}\,dX \leq CR^{-n}.
\end{align}
So, we have proved the claim \eqref{eq11.299}.
Note that we also obtain from \eqref{eq5.59}
\begin{equation}  \label{eq5.60}
\tri{(1-\zeta)\vec{v}_\rho}_{\bR^{n+1}}^2 \leq C R^{-n}
\quad\forall\rho>0.
\end{equation}
Now, let $A_\tau=\set{X\in\bR^{n+1}: \abs{D\vec{v}^\rho(X)}>\tau}$
and choose $R=\tau^{-1/(n+1)}$.
Then,
\begin{equation*}
\abs{A_\tau\setminus Q_R(Y)}\le \tau^{-2}
\int_{A_\tau\setminus Q_R(Y)}\abs{D\vec{v}^\rho}^2
\le C \tau^{-\frac{n+2}{n+1}}
\end{equation*}
and $\abs{A_\tau\cap Q_R(Y)}\le C R^{n+2}=C \tau^{-\frac{n+2}{n+1}}$.
We have thus proved \eqref{eq11.17}.
To show \eqref{eq11.29}, we first note that, for any $\tau>0$, we have
\begin{align}
\nonumber
\int_{Q_R(Y)} \abs{D\vec{v}_\rho}^p&=
\int_{Q_R(Y)\cap{\set{\abs{D\vec{v}^\rho}\le \tau}}}\abs{D\vec{v}_\rho}^p+
\int_{Q_R(Y)\cap{\set{\abs{D\vec{v}^\rho}> \tau}}}\abs{D\vec{v}_\rho}^p\\
\label{eqn:E326}
&\le  \tau^p\abs{Q_R}+
\int_{\set{\abs{D\vec{v}^\rho}> \tau}}\abs{D\vec{v}_\rho}^p.
\end{align}
By using \eqref{eq11.17} and the assumption $p\in [1,\tfrac{n+2}{n+1})$,
we estimate
\begin{align}
\nonumber
\int_{\set{\abs{D\vec{v}^\rho}> \tau}}\abs{D\vec{v}_\rho}^p
&=\int_0^\infty p t^{p-1}\abs{\set{\abs{D\vec{v}^\rho}> \max(t,\tau)}}\,dt\\
\nonumber
&\le C \tau^{-\frac{n+2}{n+1}}\int_0^\tau p t^{p-1}\,dt
+ C \int_\tau^\infty p t^{p-1-(n+2)/(n+1)}\,dt\\
\label{eqn:E327}
&=C\left(1-p/(p-\tfrac{n+2}{n+1})\right) \tau^{p-(n+2)/(n+1)}.
\end{align}
By setting $\tau=R^{-(n+1)}$ in \eqref{eqn:E326} and \eqref{eqn:E327}, we get
\begin{equation*}  
\int_{Q_R(Y)} \abs{D\vec{v}_\rho}^p\le
C R^{-(n+1)p+n+2},
\end{equation*}
from which \eqref{eq11.29} follows.
The lemma is proved.
\end{proof}

\begin{lemma} \label{lem3.3.1}
For any $Y\in\bR^{n+1}$ and any $\rho>0$, we have
\begin{equation} \label{eq11.07}
\abs{\set{X\in \bR^{n+1}:\abs{\vec{\Gamma}^\rho(X,Y)}>\tau}}
\leq C \tau^{-(n+2)/n} \quad \forall \tau>0,
\end{equation}
i.e., $\abs{\vec{\Gamma}^\rho(\cdot,Y)}$ has a weak-$L^{(n+2)/n}$ estimate.
Moreover, for any $R>0$,  the following uniform $L^p$ estimate
for $\vec\Gamma^\rho(\cdot,Y)$ holds:
\begin{equation} \label{eq11.27}
\norm{\vec{\Gamma}^\rho(\cdot,Y)}_{L^p(Q_R(Y))}\leq C_p R^{-n+(n+2)/p}
\quad\forall \rho>0 \quad \forall p\in [1,\tfrac{n+2}{n}).
\end{equation}
\end{lemma}
\begin{proof}
From \eqref{eq5.60} and \eqref{eqn:2.2} it follows that
\begin{equation*}
\norm{(1-\zeta)\vec{v}_\rho}_{L^{2+4/n}(\bR^{n+1})}\leq
\tri{(1-\zeta)\vec{v}_\rho}_{\bR^{n+1}} \leq C R^{-n/2}.
\end{equation*}
Since $\zeta\in C^\infty_c(Q_R(Y))$, we then find that for all $R>0$, we have
\begin{equation} \label{eq11.277}
\int_{\bR^{n+1}\setminus \overline Q_R(Y)}\abs{\vec{v}_\rho}{}^{2+4/n}
\leq C R^{-n-2} \quad \forall \rho>0.
\end{equation}
Let $A_\tau=\set{X\in\bR^{n+1}: \abs{\vec{v}^\rho(X)}>\tau}$
and set $R=\tau^{-1/n}$.
Then
\begin{equation*}
\abs{A_\tau\setminus Q_R(Y)}\le \tau^{-\frac{2n+4}{n}}
\int_{A_\tau\setminus Q_R(Y)}\abs{\vec{v}^\rho}{}^{2+4/n}
\le C \tau^{-\frac{2n+4}{n}} \tau^{\frac{n+2}{n}}
=C \tau^{-\frac{n+2}{n}}
\end{equation*}
and $\abs{A_\tau\cap Q_R(Y)}\le C R^{n+2}=C \tau^{-\frac{n+2}{n}}$.
We have proved \eqref{eq11.07}.
Next, by utilizing \eqref{eq11.07} instead of \eqref{eq11.17}
and proceeding as in the proof of Lemma~\ref{lem3.3.2}, we obtain
\begin{equation*}  
\int_{Q_R(Y)} \abs{\vec{v}_\rho}^p\leq C R^{-np+n+2},
\end{equation*}
from which \eqref{eq11.27} follows.
The lemma is proved.
\end{proof}

Let us fix $q\in (1,\tfrac{n+2}{n+1})$.
By \eqref{eq11.27} and \eqref{eq11.29}, we have
\begin{equation*}
\max_{1\leq i,j\leq N}
\|\Gamma^\rho_{ij}(\cdot,Y)\|_{W^{1,q}_x(Q_R^+(Y))}\leq C(R)<\infty
\quad \forall\rho>0.
\end{equation*}
By a diagonalization process, we obtain a sequence
$\{\rho_\mu\}_{\mu=1}^\infty$ tending to zero and a function
$\vec{\Gamma}(\cdot,Y)$,
which we shall call a fundamental matrix for $\cL$,
such that
\begin{equation} \label{eq11.46}
\vec{\Gamma}^{\rho_\mu}(\cdot,Y)\wto
\vec{\Gamma}(\cdot,Y)\quad\text{weakly in } W^{1,q}_x(Q_R(Y))^{N\times N}
\quad \forall R>0.
\end{equation}
Then, for any $\vec{\phi}\in C^\infty_c(\bR^{n+1})^N$, it follows from
\eqref{eq2.24}
\begin{align}
\nonumber
-\int_{\bR^{n+1}}&\Gamma_{ik}(\cdot,Y)\phi^i_t+\int_{\bR^{n+1}}
A^{\alpha\beta}_{ij}D_\beta\Gamma_{jk}(\cdot,Y)D_\alpha\phi^i\\
\nonumber
&\qquad=-\lim_{\mu\to\infty}\int_{\bR^{n+1}}
\Gamma^{\rho_\mu}_{ik}(\cdot,Y)\phi^i_t+
\lim_{\mu\to\infty}\int_{\bR^{n+1}}
A^{\alpha\beta}_{ij}D_\beta\Gamma^{\rho_\mu}_{jk}(\cdot,Y) D_\alpha \phi^i\\
\label{eqn:E-44}
&\qquad=\lim_{\mu\to\infty} \dashint_{Q^-_{\rho_\mu}(Y)}\phi^k = \phi^k(Y).
\end{align}
Let $\vec{v}_\rho$ be the $k$-th column of $\vec{\Gamma}^\rho(\cdot,y)$
as before, and let $\vec{v}$ be the corresponding $k$-th column
of $\vec{\Gamma}(\cdot,y)$.
Then, for any $\vec{g}\in L^\infty_{c}(Q_R(Y))^N$, \eqref{eq11.27} yields
\begin{equation}
\label{eqn:E-45}
\abs{\int_{Q_R(Y)}\vec{v}\cdot\vec{g}}
=\lim_{\mu\to\infty}\abs{\int_{Q_R(Y)}\vec{v}_{\rho_\mu}\cdot\vec{g}}
\le C_pR^{-n+(n+2)/p}\norm{\vec{g}}_{L^{p'}(Q_R(Y))},
\end{equation}
where $p'$ denotes the conjugate exponent of $p\in [1,\tfrac{n+2}{n})$.
Therefore, we obtain
\begin{equation}
\label{eqn:E-46}
\norm{\vec{\Gamma}(\cdot,Y)}_{L^p(Q_R(Y))}\le C_p\,R^{-n+(n+2)/p}
\quad \forall p\in [1,\tfrac{n+2}{n}).
\end{equation}
By a similar reasoning, we also have by \eqref{eq11.29}
\begin{equation}
\label{eqn:E-47}
\norm{D\vec{\Gamma}(\cdot,Y)}_{L^p(Q_R(Y))}\le C_p\,R^{-n-1+(n+2)/p}
\quad \forall p\in [1,\tfrac{n+2}{n+1}).
\end{equation}
With the aid of \eqref{eq11.277}, we also obtain
\begin{equation}
\label{eqn:E-48}
\norm{\vec{\Gamma}(\cdot,Y)}_{L^{2+4/n}(\bR^{n+1}\setminus \overline Q_R(Y))}^2
\leq C R^{-n}
\end{equation}
Also, by \eqref{eq5.60}, Lemma~\ref{lem:Ap01} in Appendix,
and the assumption $\zeta\in C^\infty_c(Q_R(Y))$, we find
(by passing to a subsequence in \eqref{eq11.46} if necessary)
\begin{equation}
\label{eqn:E-49}
\tri{\vec{\Gamma}(\cdot,Y)}^2_{\bR^{n+1}\setminus \overline Q_R(Y)}
\leq C R^{-n}.
\end{equation}
Moreover, arguing as in the proof of Lemma~\ref{lem3.3.1}
and Lemma~\ref{lem3.3.2}, we see that
the estimates \eqref{eqn:E-48} and \eqref{eqn:E-49} imply that
for all $Y\in \bR^{n+1}$, we have
\begin{gather}
\label{eqn:E-51}
\abs{\set{X\in\bR^{n+1}:\abs{\vec{\Gamma}(X,Y)}>\tau}}\le
C \tau^{-\frac{n+2}{n}},\\
\label{eqn:E-52}
\abs{\set{X\in\bR^{n+1}:\abs{D_x\vec{\Gamma}(X,Y)}>\tau}}\le
C \tau^{-\frac{n+2}{n+1}}.
\end{gather}
Next, we prove \eqref{eq5.21} in Theorem~\ref{thm1}.
Let $\eta\in C^\infty_c(\bR^{n+1})$ be such that
$\eta\equiv 1$ on $Q_r(Y)$ for some $r>0$.
Following the derivation of \eqref{eq5.60}, we find
\begin{equation}  \label{eq7.00}
\tri{(1-\eta)\vec{\Gamma}^\rho(\cdot,Y)}_{\bR^{n+1}}\le C(\eta)<\infty
\quad\text{uniformly in } \rho>0.
\end{equation}
Therefore, by passing to a subsequence if necessary, we may assume that
there is a function $\tilde{\vec\Gamma}(\cdot,Y)$ in
$\rV_2(\bR^{n+1})^{N\times N}$ such that
\begin{equation}  \label{eq7.01}
(1-\eta)\vec{\Gamma}^{\rho_\mu}(\cdot,Y)\wto
\tilde{\vec\Gamma}(\cdot,Y)
\quad\text{``very weakly'' in } V_2(\bR^{n+1})^{N\times N}.
\end{equation}
Note that \eqref{eq7.01} in particular implies that
\begin{equation*}
\lim_{\mu\to\infty}
\int_{\bR^{n+1}}(1-\eta)\vec{\Gamma}^{\rho_\mu}(\cdot,Y)\vec{\phi}
=\int_{\bR^{n+1}} \tilde{\vec\Gamma}(\cdot,Y) \vec{\phi}\quad
\forall \vec{\phi}\in C^\infty_c(\bR^{n+1})^N.
\end{equation*}
On the other hand, by \eqref{eq11.46} we find
for all $\vec{\phi}\in C^\infty_c(\bR^{n+1})^N$,
\begin{align*}
\lim_{\mu\to\infty}
\int_{\bR^{n+1}}(1-\eta)\vec{\Gamma}^{\rho_\mu}(\cdot,Y)\vec{\phi}
&=\lim_{\mu\to\infty}
\int_{\bR^{n+1}}\vec{\Gamma}^{\rho_\mu}(\cdot,Y)(1-\eta)\vec{\phi}\\
&=\int_{\bR^{n+1}} \vec\Gamma(\cdot,Y)(1-\eta)\vec{\phi}
=\int_{\bR^{n+1}}(1-\eta)\vec\Gamma(\cdot,Y)\vec{\phi}.
\end{align*}
Therefore,
$\tilde{\vec\Gamma}(\cdot,Y)\equiv(1-\eta)\vec{\Gamma}(\cdot,Y)$,
and thus $(1-\eta)\vec{\Gamma}(\cdot,Y)\in \rV_2(\bR^{n+1})^{N\times N}$.
Now, we will show
$(1-\eta)\vec{\Gamma}(\cdot,Y)\in \rV^{1,0}_2(\bR^{n+1})^{N\times N}$.
To see this, let $\vec v$ be the $k$-th column of
$\vec\Gamma(\cdot,Y)$ and set $\vec u=(1-\eta)\vec v$.
Fix $t_0$ such that $\eta\equiv 0$ on $(-\infty,t_0]\times \bR^n$.
Observe that for all $T>t_0$,
$\vec u$ is a weak solution in $\rV_2((t_0,T)\times\bR^n)^N$ of the problem
\begin{equation*}
\left\{\begin{array}{l l}
\cL \vec u=\vec f + D_\alpha \vec g_\alpha\\
\vec{u}(t_0,\cdot)= 0,\end{array}\right.
\end{equation*}
where $\vec f:=-\eta_t\vec v+D_\alpha\eta\vec{A}^{\alpha\beta}D_\beta\vec v$
and $\vec g_\alpha:=D_\beta\eta\vec{A}^{\alpha\beta}\vec v$.
Since $\eta\equiv 1$ on $Q_r(Y)$ and $\eta$ is compactly supported,
\eqref{eqn:E-49} implies that
$\vec f, \vec g_\alpha \in L^2((t_0,T)\times\Omega)^N$.
Then, by Theorem~4.2 in \cite[\S III.4]{LSU},
we have $\vec u \in \rV^{1,0}_2((t_0,T)\times\Omega)^N$.
Note that \eqref{eq7.00} implies $\tri{\vec u}_{(t_0,T)\times\bR^n}\le C(\eta)$.
Since $T>t_0$ is arbitrary, we have
$\vec u \in \rV^{1,0}_2((t_0,\infty)\times\Omega)^N$.
Moreover, as in Section~\ref{sec0301}, if we extend $\vec u$ to the
entire $\bR^{n+1}$ by setting $\vec u\equiv 0$ on $(-\infty,t_0)\times \bR^n$,
then we have $\vec u\in \rV^{1,0}_2(\bR^{n+1})^N$.
We have proved \eqref{eq5.21}.

\subsection{Continuity of the fundamental matrix}  \label{sec:3.4}
With the aid of \eqref{eqP.24}, it follows from \eqref{eqn:E-44}
and \eqref{eqn:E-49} that $\vec{\Gamma}(\cdot,Y)$
is H\"older continuous in $\bR^{n+1}\setminus\set{Y}$.
In fact, by the same reasoning, it follows
from \eqref{eq2.24} and \eqref{eq11.299} that
on any compact set $K\Subset\bR^{n+1}\setminus\set{Y}$, the sequence
$\set{\vec{\Gamma}^{\rho_\mu}(\cdot,Y)}_{\mu=1}^\infty$
is equicontinuous on $K$.
Also, by Lemma~\ref{lem3} and \eqref{eq11.277}, we find that
there are $C_K<\infty$ and $\rho_K>0$ such that
\begin{equation*}  
\|\vec{\Gamma}^{\rho}(\cdot,Y)\|_{L^\infty(K)} \leq C_K
\quad \forall \rho<\rho_K \quad
\text{for any compact } K\Subset \bR^{n+1}\setminus\set{Y}.
\end{equation*}
Therefore, we may assume, by passing to a subsequence if necessary, that
\begin{equation} \label{eq3.420}
\vec{\Gamma}^{\rho_\mu}(\cdot,Y) \rightarrow\vec{\Gamma}(\cdot,Y)
\quad\text{uniformly on } K \text{ for any compact }
K\Subset \bR^{n+1}\setminus\set{Y}.
\end{equation}

We will now show that $\vec{\Gamma}(X,\cdot)$ is also H\"older continuous
in $\bR^{n+1}\setminus\set{X}$.
Denote by ${}^t\!\vec{\Gamma}^{\sigma}(\cdot,X)$ the averaged
fundamental matrix associated to $\cLt$.
More precisely, for $1\le k \le N$ and $X=(t,x)\in \bR^{n+1}$ fixed,
let $\vec{w}_\sigma=\vec{w}_{\sigma;X,k}$
be the unique weak solution in $\rV^{1,0}_2((-\infty,t_0)\times\bR^n)^N$
of the backward problem
\begin{equation*}  
\left\{\begin{array}{l l}
\cLt \vec{u}=\frac{1}{\abs{Q^{+}_\sigma}}1_{Q^+_\sigma(X)} \vec{e}_k\\
\vec{u}(t_0,\cdot)= 0,\end{array}\right.
\end{equation*}
where $t_0>t+\sigma^2$ is fixed but arbitrary.
Then, as before, we may extend $\vec{w}_\sigma$ to the entire
$\bR^{n+1}$ by setting
$\vec{w}_\sigma\equiv 0$ on $(t+\sigma^2,\infty)\times \bR^n$
so that $\vec{w}_\sigma$ belongs to $\rV^{1,0}_2(\bR^{n+1})$ and satisfies for
all $s_1<t$ the identity
\begin{equation} \label{eq2.241}
\int_{\bR^n}w_\sigma^i \phi^i(s_1,\cdot)
+\int_{s_1}^\infty\!\!\!\int_{\bR^n}w_\sigma^i \phi^i_t
+\int_{s_1}^\infty\!\!\!\int_{\bR^n}
{}^t\!A^{\alpha\beta}_{ij}D_\beta w_\sigma^j D_\alpha \phi^i
= \dashint_{Q^+_\sigma(X)}\phi^k
\end{equation}
for all $\vec{\phi}\in C^\infty_{c,p}(\bR^{n+1})^N$.
We define the averaged fundamental matrix
${}^t\!\vec{\Gamma}^\sigma(\cdot,X)
=({}^t\!\Gamma^\sigma_{jk}(\cdot,X))_{j,k=1}^N$
for $\cLt$ by
\begin{equation*}
{}^t\!\Gamma^\sigma_{jk}(\cdot,X)=w_\sigma^j=w^j_{\sigma;Y,k}.
\end{equation*}
By a similar argument as appears in the previous subsection,
we obtain a sequence $\set{\sigma_\nu}_{\nu=1}^\infty$ tending to $0$
and a function ${}^t\!\vec{\Gamma}(\cdot,X)$,
which we shall call a fundamental matrix for $\cLt$, such that
for some $q>1$,
\begin{equation} \label{eq11.46p}
{}^t\!\vec{\Gamma}^{\sigma_\nu}(\cdot,X)\wto
{}^t\!\vec{\Gamma}(\cdot,X)\quad\text{weakly in } W^{1,q}_x(Q_R(X))^{N\times N}
\quad \forall R>0.
\end{equation}
Moreover, as in \eqref{eqn:E-44}, ${}^t\!\vec{\Gamma}(\cdot,X)$ satisfies
\begin{equation*}
\int_{\bR^{n+1}}{}^t\!\Gamma_{ik}(\cdot,X)\phi^i_t+
\int_{\bR^{n+1}}
{}^t\!A^{\alpha\beta}_{ij}D_\beta{}^t\!\Gamma_{jk}(\cdot,X)D_\alpha \phi^i
=\phi^k(X)
\end{equation*}
for all $\vec{\phi}\in C^\infty_c(\bR^{n+1})^N$.
Also, by a similar reasoning, we find that
${}^t\!\vec{\Gamma}(\cdot,X)$ satisfies all the properties
corresponding to \eqref{eqn:E-46}--\eqref{eqn:E-52} and is H\"older continuous
in $\bR^{n+1}\setminus\set{X}$.
Therefore, the following lemma will prove the claim that
$\vec{\Gamma}(X,\cdot)$ is H\"older continuous in $\bR^{n+1}\setminus\set{X}$
as well as the estimates \eqref{eq5.15}--\eqref{eq5.20}.

\begin{lemma}  \label{lem:3.5}
Let $\vec{\Gamma}(\cdot,Y)$ and ${}^t\!\vec{\Gamma}(\cdot,X)$ be fundamental
matrices of $\cL$ and $\cLt$, respectively, as constructed above. Then
\begin{equation}
\label{eqn:E-66}
\vec{\Gamma}(X,Y)={}^t\!\vec{\Gamma}(Y,X)^T \quad \forall X\neq Y.
\end{equation}
\end{lemma}
\begin{proof}
Fix $X, Y\in\bR^{n+1}$ with $X\neq Y$.
From the constructions of
$\vec{\Gamma}^\rho(\cdot,Y)$ and ${}^t\!\vec{\Gamma}^\sigma(\cdot,X)$,
it is not hard to verify that
\begin{equation}
\label{eqn:E-59}
\dashint_{Q_{\rho}^{-}(Y)}{}^t\!\Gamma_{kl}^{\sigma}(\cdot,X)
=\dashint_{Q_{\sigma}^{+}(X)}\Gamma_{lk}^{\rho}(\cdot,Y).
\end{equation}
Let $\set{\rho_\mu}_{\mu=1}^\infty$ and $\set{\sigma_\nu}_{\nu=1}^\infty$ be
approximating sequences for $\vec{\Gamma}(\cdot,Y)$ and
${}^t\!\vec{\Gamma}(\cdot,X)$, respectively, and denote
\begin{equation*}     
g^{kl}_{\mu\nu}:=
\dashint_{Q_{\rho_\mu}^{-}(Y)}{}^t\!\Gamma_{kl}^{\sigma_\nu}(\cdot,X)
=\dashint_{Q_{\sigma_\nu}^{+}(X)}\Gamma_{lk}^{\rho_\mu}(\cdot,Y).
\end{equation*}
From the continuity of $\Gamma_{lk}^{\rho_\mu}(\cdot,Y)$, it follows that
\begin{equation*}
\lim_{\nu\to\infty} g^{kl}_{\mu\nu}=
\lim_{\nu\to\infty}
\dashint_{Q_{\sigma_\nu}^{+}(X)}\Gamma_{lk}^{\rho_\mu}(\cdot,Y)=
\Gamma_{lk}^{\rho_\mu}(X,Y),
\end{equation*}
and thus, by \eqref{eq3.420}, we obtain
\begin{equation*}
\lim_{\mu\to\infty} \lim_{\nu\to\infty} g^{kl}_{\mu\nu}=
\lim_{\mu\to\infty} \Gamma_{lk}^{\rho_\mu}(X,Y)=\Gamma_{lk}(X,Y).
\end{equation*}
On the other hand, \eqref{eq11.46p} yields that
\begin{equation*}
\lim_{\nu\to\infty} g^{kl}_{\mu\nu}=
\lim_{\nu\to\infty}
\dashint_{Q_{\rho_\mu}^{-}(Y)}{}^t\!\Gamma_{kl}^{\sigma_\nu}(\cdot,X)=
\dashint_{Q_{\rho_\mu}^{-}(Y)}{}^t\!\Gamma_{kl}(\cdot,X),
\end{equation*}
and thus, it follows from the continuity of ${}^t\!\Gamma_{kl}(\cdot,X)$ that
\begin{equation*}
\lim_{\mu\to\infty} \lim_{\nu\to\infty} g^{kl}_{\mu\nu}=
\lim_{\mu\to\infty} \dashint_{Q_{\rho_\mu}^{-}(Y)}{}^t\!\Gamma_{kl}(\cdot,X)
={}^t\!\Gamma_{kl}(Y,X).
\end{equation*}
We have thus shown that
$\Gamma_{lk}(X,Y)={}^t\!\Gamma_{kl}(Y,X)$ for all $X\neq Y$.
\end{proof}

So far, we have seen that there is a sequence
$\set{\rho_\mu}_{\mu=1}^\infty$ tending to $0$
such that $\vec{\Gamma}^{\rho_\mu}(\cdot,Y) \to \vec{\Gamma}(\cdot,Y)$
in $\bR^{n+1}\setminus\set{Y}$.
However, by \eqref{eqn:E-59}, we obtain for $X\neq Y$,
\begin{align*}
\Gamma_{lk}^{\rho}(X,Y)&=
\lim_{\nu\to\infty}\dashint_{Q_{\sigma_\nu}^{+}(X)}\Gamma_{lk}^{\rho}(\cdot,Y)
=\lim_{\nu\to\infty}
\dashint_{Q_{\rho}^{-}(Y)}{}^t\!\Gamma_{kl}^{\sigma_\nu}(\cdot,X)\\
&=\dashint_{Q_{\rho}^-(Y)}{}^t\!\Gamma_{kl}(\cdot,X)
=\dashint_{Q_{\rho}^{-}(Y)}\Gamma_{lk}(X,\cdot).
\end{align*}
Therefore, we have the following representation of
the averaged fundamental matrix:
\begin{equation}
\label{eqn:E-68}
\vec{\Gamma}^\rho(X,Y)=\dashint_{Q_\rho^-(Y)} \vec{\Gamma}(X,Z)\,dZ
\quad \forall X\neq Y.
\end{equation}
In particular, by using the continuity of $\vec{\Gamma}(X,\cdot)$, we obtain
\begin{equation}
\label{eq3.786}
\lim_{\rho\to 0} \vec{\Gamma}^\rho(X,Y)=\vec{\Gamma}(X,Y)
\quad \forall X\neq Y,
\end{equation}
and thus, from \eqref{eq11.16} that
\begin{equation} \label{eq3.6.0}
\abs{\vec{\Gamma}(X,Y)}\leq C |X-Y|_p^{-n} \quad\forall X\neq Y.
\end{equation}
Furthermore, due to \eqref{eq11.54}, it also follows that
\begin{equation}  \label{eq11.55}
\vec{\Gamma}(\cdot,\cdot,s,y)\equiv 0\quad\text{in}\quad(-\infty,s)\times \bR^n.
\end{equation}

\subsection{Representation formulas}   \label{sec3.5}
We shall prove the identity \eqref{eqn:E-70}.
For a given $\vec{f}\in C^\infty_c(\bR^{n+1})^N$,
fix $t_0$ such that $\vec{f}\equiv 0$ on $(-\infty,t_0]\times\bR^n$.
Let $\vec{u}$ be the unique weak solution in
$\rV^{1,0}_2((t_0,\infty)\times\bR^n)^N$ of the problem
\begin{equation*}
\left\{\begin{array}{l l}
\cL \vec{u}=\vec{f}\\
\vec{u}(t_0,\cdot)= 0.\end{array}\right.
\end{equation*}
By the uniqueness, $\vec{u}$ does not depend on the choice of $t_0$
and we may extend $\vec{u}$ to $\bR^{n+1}$ by setting
$\vec{u}\equiv 0$ on $(-\infty,t_0)\times \bR^n$ so that
$\vec{u}\in \rV^{1,0}_2(\bR^{n+1})^N$.
Let ${}^t\!\vec{\Gamma}^\sigma(\cdot,X)$ be the averaged fundamental matrix
for $\cLt$.
Then, as in \eqref{eq2.17}, we have
\begin{equation}  \label{eq:z-001}
\int_{\bR^{n+1}}{}^t\!\Gamma^{\sigma_\nu}_{ik}(\cdot,X)f^i=
\dashint_{Q_{\sigma_\nu}^+(X)}u^k.
\end{equation}
As in \eqref{eq10.191}, the property (PH) implies that
$\vec{u}$ is continuous.
Therefore, by taking the limit $\nu\to\infty$ in \eqref{eq:z-001} and
then using \eqref{eqn:E-66},
we find that
\begin{equation*}
\int_{\bR^{n+1}}\Gamma_{ki}(X,\cdot)f^i= u^k(X)
\end{equation*}
or equivalently, in terms of matrix multiplication
\begin{equation*}     
\vec{u}(X)= \int_{\bR^{n+1}}\vec{\Gamma}(X,Y)\vec{f}(Y)\,dY.
\end{equation*}
We have thus proved \eqref{eqn:E-70}.

Next, we shall prove the identity \eqref{eq5.24}.
Let $\vec{g}\in L^2(\bR^n)^N$ be given and
let $\vec{u}$ be the weak solution in $\rV^{1,0}_2((s,\infty)\times\bR^n)^N$
of the problem \eqref{eq5.88}.
Fix $X=(t,x)\in\bR^{n+1}$ with $t>s$ and let
$\vec{w}_\sigma=\vec{w}_{\sigma;X,k}$ be
the $k$-th column of the averaged fundamental matrix
of the transpose operator $\cLt$ with a pole at $X$.
By using \eqref{eq2.241} and \eqref{eq:c01} together with standard
approximation techniques (see e.g., \cite[\S III.2]{LSU}),
we find that for $\sigma_\nu$ sufficiently small,
\begin{equation}  \label{eqn:3.56}
\dashint_{Q^+_{\sigma_\nu}(X)} u^k= \int_{\bR^n} \vec{w}_{\sigma_\nu}(s,y)\cdot
\vec{g}(y)\,dy.
\end{equation}
Let us assume for the moment that $\vec g$ is compactly supported.
Then, by estimates similar to \eqref{eq3.420} and \eqref{eq3.6.0},
\begin{equation*} 
\lim_{\nu\to \infty}\int_{\bR^n}\vec{w}_{\sigma_\nu}(s,y)\cdot \vec{g}(y)\,dy
=\int_{\bR^n} \vec{w}(s,y)\cdot \vec{g}(y)\,dy,
\end{equation*}
where $\vec{w}$ is the $k$-th column of ${}^t\!\vec{\Gamma}(\cdot,X)$.
Also, by the property (PH), we find that $\vec{u}$ is continuous at $X$.
Therefore, by taking the limit $\nu\to\infty$ in \eqref{eqn:3.56}, we obtain
\begin{equation}  \label{eqn:3.57}
u^k(X)= \int_{\bR^n} {}^t\!\Gamma_{ik}(s,y,t,x)g^i(y)\,dy
=\int_{\bR^n} \Gamma_{ki}(t,x,s,y)g^i(y)\,dy.
\end{equation}
We have thus derived \eqref{eq5.24} under the additional assumption
that $\vec g$ is compactly supported.
For general $\vec g\in L^2(\bR^n)^N$, let $\{\vec g_m\}_{m=1}^\infty$ be
a sequence in $L^\infty_c(\bR^n)^N$ such that
$\lim_{m\to\infty}\norm{\vec g_m-\vec g}_{L^2(\bR^n)}=0$.
Let $\vec{u}_m$ be the unique weak solution in
$\rV^{1,0}_2((s,\infty)\times\bR^n)^N$
of the problem \eqref{eq5.88} with $\vec g$ replaced by $\vec g_m$.
Then, by the energy inequality, we find that
$\lim_{m\to\infty}\tri{\vec u_m-\vec u}_{(s,t)\times\bR^n}=0$.
Therefore, by Lemma~\ref{lem3}, we conclude
$\lim_{m\to\infty}\abs{\vec u_m(X)-\vec u(X)}=0$.
On the other hand, by \eqref{eq5.21} and \eqref{eqn:E-49}, we have 
$\norm{\vec\Gamma(t,x,s,\cdot)}_{L^2(\bR^n)}\leq C(t-s)^{-n/2}$, and thus
\begin{equation*}
\lim_{m\to\infty}
\int_{\bR^n}\vec\Gamma(t,x,s,y)\vec{g}_m(y)\,dy=
\int_{\bR^n}\vec\Gamma(t,x,s,y)\vec{g}(y)\,dy.
\end{equation*}
This completes the proof of the representation formula \eqref{eq5.24}.

Next, for $\vec{g}\in L^2(\bR^n)^N$,
let $\vec{u}$ be the unique weak solution in
$\rV^{1,0}_2((-\infty,t)\times\bR^n)^N$ of the backward problem
\begin{equation} \label{eq3.15}
\left\{\begin{array}{l l}
\cLt \vec{u}=0\\
\vec{u}(t,\cdot)= \vec{g}.\end{array}\right.
\end{equation}
Then, by a similar reasoning as above,
$\vec{u}$ has the following representation
\begin{equation} \label{eq3.16}
\vec{u}(s,y)=
\int_{\bR^n} {}^t\!\vec{\Gamma}(s,y,t,x)\vec{g}(x)\,dx.
\end{equation}

\subsection{Uniqueness of fundamental matrix}   \label{sec:3.6}
We have already shown that $\vec{\Gamma}(X,Y)$
is continuous in $\set{(X,Y)\in\bR^{n+1}\times\bR^{n+1}:X\neq Y}$ and
satisfies $\vec{\Gamma}(t,x,s,y)\equiv 0$ for $t<s$.
Also, we have seen that
$\vec\Gamma(X,\cdot)$ is locally integrable in $\bR^{n+1}$
for all $X\in\bR^{n+1}$ and that for all
$\vec{f} \in C^\infty_c(\bR^{n+1})^N$, the function $\vec{u}$ given by
\eqref{eqn:E-70} belongs to the space $\rV^{1,0}_2(\bR^{n+1})^N$ and
satisfies $\cL\vec{u}=\vec{f}$ in the sense of \eqref{eqn:E-71}.
Suppose $\tilde{\vec\Gamma}(X,Y)$ is another fundamental matrix satisfying
all the properties stated above.
Then, for a given $\vec{f}\in C^\infty_c(\bR^{n+1})^N$,
let $\vec{u}$ be defined as in \eqref{eqn:E-70} and set
\begin{equation*}
\tilde{\vec{u}}(X):=\int_{\bR^{n+1}} \tilde{\vec\Gamma}(X,Y)\vec{f}(Y)\,dY.
\end{equation*}
Since $\vec\Gamma(t,x,s,y)=\tilde{\vec\Gamma}(t,x,s,y)\equiv 0$
for $t<s$ and
$\vec{f}$ is compactly supported in $\bR^{n+1}$,
we find $\vec{u}=\tilde{\vec{u}}\equiv 0$ on
$(-\infty,t_0]\times\bR^n$ for some $t_0$.
Therefore, $\vec{v}:=\vec{u}-\tilde{\vec{u}}$ is in the space
$\rV^{1,0}_2((t_0,\infty)\times\bR^n)^N$ and satisfies $\cL \vec{v}=0$ in
$(t_0,\infty)\times\bR^n$ and $\vec{v}(t_0,\cdot)=0$.
By the uniqueness (see \cite[\S III.3]{LSU}), we must have $\vec{v}\equiv 0$.
Therefore, we find
\begin{equation*}
\int_{\bR^{n+1}}(\vec{\Gamma}-\tilde{\vec\Gamma})(X,Y)\vec{f}(Y)\,dY=0
\quad \forall \vec{f}\in C^\infty_c(\bR^{n+1})^N,
\end{equation*}
and thus it follows that
$\vec{\Gamma}(X,\cdot)\equiv\tilde{\vec\Gamma}(X,\cdot)$
in $\bR^{n+1}\setminus\set{X}$.
Since $X\in\bR^{n+1}$ is arbitrary, we have proved that
$\vec{\Gamma}\equiv\tilde{\vec\Gamma}$ in
$\set{(X,Y)\in\bR^{n+1}\times\bR^{n+1}:X\neq Y}$.
\subsection{Conclusion}
We completed the proof of Theorem~\ref{thm1} in the case when
$\Omega=\bR^n$ and $R_c=\infty$, modulo the identity \eqref{eq5.23}.
We defer the proof of \eqref{eq5.23} to Section~\ref{sec:4.4}.

\mysection{Proof of Theorem~\ref{thm1}: General cases}   \label{sec4}
In this section, we prove Theorem \ref{thm1} under a general assumption
that $\Omega$ is an arbitrary open connected set in $\bR^n$
and $R_c\in (0,\infty]$.
Since we allow the case $\Omega\neq\bR^n$, we shall write $\vec{G}(X,Y)$
instead of $\vec \Gamma(X,Y)$ for the Green's matrix in $\U=\bR\times\Omega$
to avoid confusion.
To construct Green's matrix in $\U$,
we need to adjust arguments from the previous section a little bit.

Throughout this section, we shall use the notations
$\U_\rho^-(X)=Q_\rho^-(X)\cap \U$, $\U_\rho^+(X)=Q_\rho^+(X)\cap \U$, and
$\U_\rho(X)=Q_\rho(X)\cap \U$.
Also, recall that we have defined $d_X=\dist(X,\partial\U)=\dist(x,\partial\Omega)$
and $\bar{d}_X=\min(d_X,R_c)$.
As in the previous section, we employ the letter $C$ to denote a constant
depending on $n$, $N$, $\lambda$, $\Lambda$, $\mu_0$, $C_0$, and sometimes 
on an exponent $p$ characterizing Lebesgue classes.

\subsection{Averaged Green's matrix}
Let $Y=(s,y)\in \U$ and $1\le k \le N$ be fixed.
For each $\rho>0$, fix $s_0\in (-\infty,s-\rho^2)$.
We consider the problem
\begin{equation} \label{eq1.39p}
\left\{\begin{array}{l l}
\cL \vec{u}=\frac{1}{\abs{\U^-_\rho(Y)}}1_{\U^-_\rho(Y)} \vec{e}_k\\
\vec{u}(s_0,\cdot)= 0.\end{array}\right.
\end{equation}
By Lemma~\ref{lem13.1}, there is a unique weak solution
$\vec{v}_\rho=\vec{v}_{\rho;Y,k}$ in $\rV^{1,0}_2((s_0,\infty)\times\Omega)^N$
of the problem \eqref{eq1.39p}.
As in Section~\ref{sec3.2}, we may extend $\vec{v}_\rho$ to entire $\U$
by setting
\begin{equation} \label{eq11.54p}
\vec{v}_\rho\equiv 0 \text{ on } (-\infty,s-\rho^2)\times \Omega
\end{equation}
so that $\vec{v}_\rho\in \rV^{1,0}_2(\U)^N$ and satisfies for
all $t_1>s$ the identity
\begin{equation} \label{eq2.24p}
\int_{\Omega}v_\rho^i \phi^i(t_1,\cdot)
-\int_{-\infty}^{t_1}\int_{\Omega}v_\rho^i \phi^i_t
+\int_{-\infty}^{t_1}\int_{\Omega}
A^{\alpha\beta}_{ij}D_\beta v_\rho^j D_\alpha \phi^i
= \dashint_{\U^-_\rho(Y)}\phi^k
\end{equation}
for all $\vec{\phi}\in C^\infty_{c,p}(\U)^N$.
We then define the \textit{averaged Green's matrix}
$\vec{G}^\rho(\cdot,Y)=(G^\rho_{jk}(\cdot,Y))_{j,k=1}^N$ for $\cL$ in $\U$
as
\begin{equation*}  
G^\rho_{jk}(\cdot,Y)=v_\rho^j=v^j_{\rho;Y,k}.
\end{equation*}

Next, for each $\vec{f}\in L^\infty_c(\U)^N$,
fix $t_0$ such that $\vec{f}\equiv 0$ on $[t_0,\infty)\times \Omega$
and let $\vec{u}$ be the unique weak solution in
$\rV^{1,0}_2((-\infty,t_0)\times\Omega)^N$ of the problem \eqref{eq2.05}.
Again, we may extend $\vec{u}$ to $\U$ by setting
$\vec{u}\equiv 0$ on $(t_0,\infty)\times\Omega$ so that
$\vec{u}\in \rV^{1,0}_2(\U)^N$ and satisfies for all $t_1$ the identity
\begin{equation} \label{eq2.35p}
\int_{\Omega}u^i \phi^i(t_1,\cdot)
+\int_{t_1}^\infty\!\!\int_{\Omega}u^i \phi^i_t
+\int_{t_1}^\infty\!\!\int_{\Omega}
{}^t\!A^{\alpha\beta}_{ij}D_\beta u^j D_\alpha \phi^i
=\int_{t_1}^\infty\!\!\int_{\Omega} f^i \phi^i
\end{equation}
for all $\vec{\phi}\in C^\infty_{c,p}(\U)$.

We note that Lemma \ref{lem3.1.1} remains valid with
$\Omega$, $\U$, and $\U_\rho^-$ in place of $\bR^n$, $\bR^{n+1}$,
and $Q_\rho^-$, respectively.
By following a similar argument as in Section~\ref{sec3.2}, we find that
if $\vec{f}$ is supported in $Q^+_R(X_0)$,
where $X_0\in \U$ and $R< \bar{d}_{X_0}$, then we have
\begin{equation*}
\|\vec{u}\|_{L^\infty(Q_{R/4}^+(X_0))}
\leq C R^{2-(n+2)/p}\|\vec{f}\|_{L^p(Q_R^+(X_0))}
\quad \forall p>(n+2)/2.
\end{equation*}
Consequently, as in \eqref{eq11.09}, if we have $\rho<R/4$
and $Q^-_\rho(Y)\subset Q^+_{R/4}(X_0)$, then
\begin{equation*}
\|\vec{v}_\rho\|_{L^q(Q^+_R(X_0))}\leq CR^{-n+(n+2)/q}
\quad \forall q\in [1, (n+2)/n).
\end{equation*}
Therefore, by following the proof of Lemma \ref{lem3.2.5},
we conclude that
\begin{equation} \label{eq11.16p}
\abs{\vec{G}^\rho(X,Y)}\leq C |X-Y|_p^{-n}
\quad \forall\rho\leq |X-Y|_p/3
\quad\text{if }|X-Y|_p<\bar{d}_Y/2.
\end{equation}

\subsection{Construction of the Green's matrix} \label{sec4.2}
As in \eqref{eq11.299}, the estimate \eqref{eq11.16p} yields
\begin{equation} \label{eq11.299p}
\tri{\vec{G}^\rho(\cdot,Y)}_{\U\setminus \overline Q_R(Y)}^2 \leq C R^{-n}
\quad \forall \rho>0 \quad \forall R<\bar{d}_Y/2,
\end{equation}
which in turn gives the following weak-$L^{(n+2)/(n+1)}$ estimate
as in \eqref{eq11.17}:
\begin{equation} \label{eq11.17p}
\abs{\set{X\in \U:\abs{D_x\vec{G}^\rho(X,Y)}>\tau}} \leq C
\tau^{-\frac{n+2}{n+1}} \quad \forall \tau> (\bar{d}_Y/2)^{-n}.
\end{equation}
Then, as in \eqref{eq11.29}, we have for any $\rho>0$,
\begin{equation} \label{eq11.29p}
\norm{D\vec{G}^\rho(\cdot,Y)}_{L^p(Q_R(Y))} \leq C_p R^{-n-1+(n+2)/p}
\quad \forall R < \bar{d}_Y \quad \forall p\in [1,\tfrac{n+2}{n+1}).
\end{equation}
Also, as in \eqref{eq11.277} we have
\begin{equation}
\label{eq11.277p}
\int_{\U\setminus \overline Q_R(Y)}\abs{\vec{G}^\rho(\cdot,Y)}^{2+4/n}\leq
C R^{-n-2}
\quad \forall \rho>0 \quad \forall R<\bar{d}_Y/2,
\end{equation}
which gives the following weak-$L^{(n+2)/n}$ estimate as in \eqref{eq11.07}:
\begin{equation} \label{eq11.07p}
\abs{\set{X\in \U:\abs{\vec{G}^\rho(X,Y)}>\tau}} \leq C \tau^{-(n+2)/n}
\quad \forall \tau> (\bar{d}_Y/2)^{-n}.
\end{equation}
Then, as in \eqref{eq11.27}, we have for any $\rho>0$
\begin{equation} \label{eq11.27p}
\norm{\vec{G}^\rho(\cdot,Y)}_{L^p(Q_R(Y))}\leq C_p R^{-n+(n+2)/p}
\quad\forall R< \bar{d}_Y \quad \forall p\in [1,(n+2)/n).
\end{equation}

Fix $q\in (1,\frac{n+2}{n+1})$ and observe that \eqref{eq11.29p}
and \eqref{eq11.27p} imply
\begin{equation}
\label{eqn:G-23}
\norm{\vec{G}^\rho(\cdot\,Y)}_{W^{1,q}_x(Q_{\bar{d}_Y}(Y))}\leq
C(\bar{d}_Y)<\infty \quad\text{uniformly in }\rho>0.
\end{equation}
Also, fix a cut-off function
\begin{equation}  \label{eq8.01}
\zeta_Y\in C^\infty_c(Q_{\bar{d}_Y/2}(Y))\quad
\text{satisfying}\quad\zeta_Y\equiv 1\text{ on }Q_{\bar{d}_Y/4}(Y).
\end{equation}
Then, as in \eqref{eq7.00}, we find
\begin{equation}  \label{eq8.00}
\tri{(1-\zeta_Y)\vec G^\rho(\cdot,Y)}_{\U}\le C(\zeta_Y)<\infty
\quad\text{uniformly in } \rho>0.
\end{equation}
Therefore, we conclude from \eqref{eqn:G-23} and \eqref{eq8.00}
that there exist a sequence $\set{\rho_\mu}_{\mu=1}^\infty$ tending to $0$ and
functions $\vec{G}(\cdot,Y)\in W^{1,q}_x(Q_{\bar{d}_Y}(Y))^{N\times N}$
and $\tilde{\vec{G}}(\cdot,Y)\in \rV_2(\U)^{N\times N}$
such that as $\mu\to \infty$,
\begin{align}
\label{eqn:G-24}
\vec{G}^{\rho_\mu}(\cdot,Y)\wto \vec{G}(\cdot,Y)
\quad&\text{weakly in } W^{1,q}_x(Q_{\bar{d}_Y}(Y))^{N\times N}\\
\label{eqn:G-25} (1-\zeta_Y)\vec{G}^{\rho_\mu}(\cdot,Y)\wto
\tilde{\vec{G}}(\cdot,Y) \quad&\text{``very weakly'' in } V_2(\U)^{N\times N}.
\end{align}
Since we must have $\vec{G}(\cdot,Y)\equiv\tilde{\vec{G}}(\cdot,Y)$
on $Q_{\bar{d}_Y}(Y)\setminus Q_{\bar{d}_Y/2}(Y)$,
we may extend $\vec{G}(\cdot,Y)$ to the entire $\U$
by setting $\vec{G}(\cdot,Y)=\tilde{\vec{G}}(\cdot,Y)$
on $\U\setminus Q_{\bar{d}_Y}(Y)$.
We shall call this extended function a Green's matrix for $\cL$ in
$\U$ with a pole at $Y$ and still denote it by $\vec{G}(\cdot,Y)$.

Now, we prove the identity \eqref{eq5.22} in Theorem~\ref{thm1}.
Write $\vec\phi=\eta\vec\phi+(1-\eta)\vec\phi$, where
$\eta\in C^\infty_c(Q_{\bar{d}_Y}(Y))$ satisfying
$\eta\equiv 1$ on $Q_{\bar{d}_Y/2}(Y)$.
Then, by \eqref{eq2.24p}, \eqref{eqn:G-24}, and \eqref{eqn:G-25}, we obtain
\begin{align*}
\phi^k(Y)&=\lim_{\mu\to\infty} \dashint_{\U_{\rho_\mu}(Y)}\eta\phi^k +
\lim_{\mu\to\infty} \dashint_{\U_{\rho_\mu}(Y)}(1-\eta)\phi^k\\
&=\lim_{\mu\to\infty}\int_{\U}\left(
-G_{ik}^{\rho_\mu}(\cdot,Y) D_t(\eta\phi^i)+A^{\alpha\beta}_{ij}
D_\beta G^{\rho_\mu}_{jk}(\cdot,Y) D_\alpha (\eta \phi^i)\right)\\
&\qquad+\lim_{\mu\to\infty}\int_{\U}\left(
-G_{ik}^{\rho_\mu}(\cdot,Y)D_t((1-\eta)\phi^i)+A^{\alpha\beta}_{ij}
D_\beta G^{\rho_\mu}_{jk}(\cdot,Y) D_\alpha ((1-\eta) \phi^i)\right)\\
&=\int_{\U}\left(-G_{ik}(\cdot,Y) D_t(\eta\phi^i)+A^{\alpha\beta}_{ij}
D_\beta G_{jk}(\cdot,Y) D_\alpha (\eta \phi^i)\right)\\
&\qquad+\int_{\U}\left(-G_{ik}(\cdot,Y)D_t((1-\eta)\phi^i)
+A^{\alpha\beta}_{ij}D_\beta G_{jk}(\cdot,Y) D_\alpha((1-\eta)\phi^i)\right)\\
&=\int_{\U}\left(-G_{ik}(\cdot,Y) D_t\phi^i+A^{\alpha\beta}_{ij}
D_\beta G_{jk}(\cdot,Y) D_\alpha \phi^i\right)\quad \text{as desired.}
\end{align*}
By proceeding similarly as in Section~\ref{sec0303},
we derive the following estimates
which correspond to \eqref{eqn:E-46}--\eqref{eqn:E-52}:
\begin{gather}
\label{eqn:G-32}
\norm{\vec{G}(\cdot,Y)}_{L^p(Q_R(Y))}\le C_p\,R^{-n+(n+2)/p}
\quad \forall R<\bar{d}_Y \quad \forall p\in [1,\tfrac{n+2}{n}),\\
\label{eqn:G-32b}
\norm{D\vec{G}(\cdot,Y)}_{L^p(Q_R(Y))}\le C_p\,R^{-n-1+(n+2)/p}
\quad \forall R<\bar{d}_Y \quad \forall p\in [1,\tfrac{n+2}{n+1}),\\
\label{eqn:G-33}
\norm{\vec{G}(\cdot,Y)}_{L^{2+4/n}(\U\setminus \overline Q_R(Y))}^2
\le C R^{-n} \quad \forall R<\bar{d}_Y/2,\\
\label{eqn:G-34}
\tri{\vec{G}(\cdot,Y)}_{\U\setminus \overline Q_R(Y)}^2
\le C R^{-n} \quad \forall R<\bar{d}_Y/2,\\
\label{eqn:G-35}
\abs{\set{X\in\U:\abs{\vec{G}(X,Y)}>\tau}}\le C \tau^{-\frac{n+2}{n}}
\quad\forall \tau>(\bar{d}_Y/2)^{-n},\\
\label{eqn:G-36}
\abs{\set{X\in\U:\abs{D_x\vec{G}(X,Y)}>\tau}}\le C \tau^{-\frac{n+2}{n+1}}
\quad\forall \tau> (\bar{d}_Y/2)^{-n}.
\end{gather}
Next, we prove \eqref{eq5.21} in Theorem~\ref{thm1}.
Let $\eta\in C^\infty_c(\U)$ such that
$\eta\equiv 1$ on $Q_r(Y)$ for some $r<d_Y$.
Since $\bar{d}_Y\leq d_Y$, we may assume that $r<\bar{d}_Y$.
Then, as in \eqref{eq7.01}, we may assume that
there is a function $\hat{\vec G}(\cdot,Y)$ in $\rV_2(\U)^{N\times N}$
such that
\begin{equation}  \label{eq9.01}
(1-\eta)\vec{G}^{\rho_\mu}(\cdot,Y)\wto
\hat{\vec G}(\cdot,Y)
\quad\text{``very weakly'' in } V_2(\U)^{N\times N}.
\end{equation}
Note that \eqref{eq9.01} in particular implies that
\begin{equation*}
\lim_{\mu\to\infty}
\int_{\U}(1-\eta)\vec{G}^{\rho_\mu}(\cdot,Y)\vec{\phi}
=\int_{\U} \hat{\vec G}(\cdot,Y) \vec{\phi}\quad
\forall \vec{\phi}\in C^\infty_c(\U).
\end{equation*}
On the other hand, by \eqref{eqn:G-24} and \eqref{eqn:G-25}, we find that
for all $\vec{\phi}\in C^\infty_c(\U)$,
\begin{align*}
&\lim_{\mu\to\infty}
\int_{\U}(1-\eta)\vec{G}^{\rho_\mu}(\cdot,Y)\vec{\phi}\\
&\quad=\lim_{\mu\to\infty}\left(
\int_{Q_{\bar{d}_Y}(Y)}\vec{G}^{\rho_\mu}(\cdot,Y)(1-\eta)\vec{\phi}
+\int_{\U\setminus Q_{\bar{d}_Y}(Y)}
(1-\zeta_Y)\vec{G}^{\rho_\mu}(\cdot,Y)(1-\eta)\vec{\phi}\right)\\
&\quad =\int_{Q_{\bar{d}_Y}(Y)} \vec G(\cdot,Y)(1-\eta)\vec{\phi}
+\int_{\U\setminus Q_{\bar{d}_Y}(Y)}\vec{G}(\cdot,Y)(1-\eta)\vec{\phi}
=\int_{\U}(1-\eta)\vec G(\cdot,Y)\vec{\phi}.
\end{align*}
Therefore, we must have
$\hat{\vec G}(\cdot,Y)\equiv(1-\eta)\vec G(\cdot,Y)$,
and thus $(1-\eta)\vec G(\cdot,Y) \in \rV_2(\U)$.
Then, by following a similar argument as in Section~\ref{sec0303}, we find
that $(1-\eta)\vec G(\cdot,Y)$ belongs to $\rV^{1,0}_2(\U)$, which
proves \eqref{eq5.21}.

\subsection{Continuity of the Green's matrix}
Proceeding as in Section~\ref{sec:3.4},
we find that $\vec{G}(\cdot,Y)$ is H\"older continuous in $\U\setminus\set{Y}$.
Also, we may assume that
\begin{equation}  \label{eq3.420p}
\vec{G}^{\rho_\mu}(\cdot,Y) \rightarrow\vec{G}(\cdot,Y)
\text{ uniformly on $K$, for any compact } K\Subset \U\setminus\set{Y}.
\end{equation}
Denote by ${}^t\!\vec{G}^{\sigma}(\cdot,X)$ the averaged
Green's matrix associated to $\cLt$ with a pole at $X\in\U$.
More precisely, for $1\le k \le N$ and $X=(t,x)\in \U$ fixed,
let $\vec{w}_\sigma=\vec{w}_{\sigma;X,k}$
be the unique weak solution in $\rV^{1,0}_2((-\infty,t_0)\times\Omega)^N$
of the backward problem
\begin{equation*}  
\left\{\begin{array}{l l}
\cLt \vec{u}=\frac{1}{\abs{\U^+_\sigma(X)}}1_{\U^+_\sigma(X)} \vec{e}_k\\
\vec{u}(t_0,\cdot)= 0,\end{array}\right.
\end{equation*}
where $t_0>t+\sigma^2$ is fixed but arbitrary.
Then, as before, we may extend $\vec{w}_\sigma$ to the entire $\U$
by setting
$\vec{w}_\sigma\equiv 0$  on $(t+\sigma^2,\infty)\times \Omega$
so that $\vec{w}_\sigma$ belongs to $\rV^{1,0}_2(\U)$ and satisfies for
all $s_1<t$ the identity
\begin{equation} \label{eq2.241p}
\int_{\Omega}w_\sigma^i \phi^i(s_1,\cdot)
+\int_{s_1}^\infty\!\!\!\int_{\Omega}w_\sigma^i \phi^i_t
+\int_{s_1}^\infty\!\!\!\int_{\Omega}
{}^t\!A^{\alpha\beta}_{ij}D_\beta w_\sigma^j D_\alpha \phi^i
= \dashint_{\U^+_\sigma(X)}\phi^k
\end{equation}
for all $\vec{\phi}\in C^\infty_{c,p}(\U)^N$.
The averaged Green's matrix
${}^t\!\vec{G}^\sigma(\cdot,X)
=({}^t\!G^\sigma_{jk}(\cdot,X))_{j,k=1}^N$ for $\cLt$ in $\U$ is
then defined by
\begin{equation*}
{}^t\!G^\sigma_{jk}(\cdot,X)=w_\sigma^j=w^j_{\sigma;Y,k}.
\end{equation*}
By a similar argument as appears in the previous subsection,
we obtain a sequence $\set{\sigma_\nu}_{\nu=1}^\infty$ tending to $0$
and functions
${}^t\!\vec{G}(\cdot,X)\in W^{1,q}_x(Q_{\bar{d}_X}(X))^{N\times N}$
for some $q>1$ and 
${}^t\!\tilde{\vec G}(\cdot,X)\in \rV_2(\U)^{N\times N}$
such that as $\nu\to\infty$
\begin{align}
\label{eqn:G-24p}
{}^t\!\vec{G}^{\sigma_\nu}(\cdot,X)&\wto {}^t\!\vec{G}(\cdot,X)
\text{ weakly in }W^{1,q}_x(Q_{\bar{d}_X}(X))^{N\times N},\\
\label{eqn:G-25p}
(1-\zeta_X){}^t\!\vec{G}^{\sigma_\nu}(\cdot,X)&\wto
{}^t\!\tilde{\vec G}(\cdot,X) \text{ ``very weakly'' in } V_2(\U)^{N\times N},
\end{align}
where $\zeta_X$ is given similarly as in \eqref{eq8.01}.
Since we must have
${}^t\!\vec{G}(\cdot,X)\equiv {}^t\!\tilde{\vec{G}}(\cdot,X)$
on $Q_{\bar{d}_X}(X)\setminus Q_{\bar{d}_X/2}(X)$,
we may extend ${}^t\!\vec{G}(\cdot,X)$ to the entire $\U$
by setting ${}^t\!\vec{G}(\cdot,X)=\tilde{\vec{G}}(\cdot,X)$
on $\U\setminus Q_{\bar{d}_X}(X)$.
We shall call this extended function a Green's matrix for $\cLt$ in
$\U$ with a pole at $X$ and still denote it by ${}^t\!\vec{G}(\cdot,X)$.
Then, by the same reasoning as in the previous subsection,
we find that ${}^t\!\vec{G}(\cdot,X)$ is
H\"older continuous in $\U\setminus\set{X}$ and it satisfies
\begin{equation*}
\int_{\U}\left({}^t\!G_{ik}(\cdot,X)\phi^i_t+
{}^t\!A^{\alpha\beta}_{ij}D_\beta{}^t\!G_{jk}(\cdot,X)D_\alpha \phi^i \right)
=\phi^k(X) \quad\forall\vec\phi\in C^\infty_c(\U)^N
\end{equation*}
and all the properties corresponding to \eqref{eqn:G-32}--\eqref{eqn:G-36}.
Therefore, the estimates \eqref{eq5.15}--\eqref{eq5.20} will follow once
we establish the following identity, the proof of which is essentially given
in Lemma~\ref{lem:3.5}:
\begin{equation}
\label{eqn:E-66p}
\vec{G}(X,Y)={}^t\!\vec{G}(Y,X)^T \quad \forall X,Y\in \U,\quad X\neq Y.
\end{equation}
Moreover, \eqref{eqn:E-66p} also implies that
$\vec{G}(X,\cdot)$ is H\"older continuous in $\U\setminus\set{X}$.
Furthermore, as in \eqref{eqn:E-68} and \eqref{eq3.786} we have
\begin{equation*}  
\vec{G}^\rho(X,Y)=\dashint_{\U_\rho^-(Y)} \vec{G}(X,Z)\,dZ
\quad \forall X,Y\in\U,\quad X\neq Y\quad\text{and}
\end{equation*}
\begin{equation*}  
\lim_{\rho\to 0} \vec{G}^\rho(X,Y)=\vec{G}(X,Y)
\quad \forall X,Y\in \U,\quad X\neq Y.
\end{equation*}
Then, from \eqref{eq11.16p} and \eqref{eqn:E-66p}, we find
\begin{equation} \label{eq3.6.0p}
\abs{\vec{G}(X,Y)}\leq C |X-Y|_p^{-n}\quad
\text{if}\quad 0<|X-Y|_p<\tfrac{1}{2}\max(\bar{d}_X,\bar{d}_Y)
\end{equation}
and also by \eqref{eq11.54p}, we obtain
\begin{equation}  \label{eq11.55p}
\vec{G}(\cdot,\cdot,s,y)\equiv 0\quad\text{in}\quad(-\infty,s)\times \Omega.
\end{equation}

\subsection{Proof of the identity \eqref{eq5.23}}  \label{sec:4.4}
Let $\vec{g}\in L^2(\Omega)^N$ be given and
let $\vec{u}$ be the unique weak solution in
$\rV^{1,0}_2((s,\infty)\times\Omega)^N$ of the problem \eqref{eq5.88}.
By proceeding similarly as in Section~\ref{sec3.5} we derive
the following representation for $\vec{u}(t,x)$:
\begin{equation}  \label{eqn:3.58p}
\vec{u}(t,x)=\int_{\Omega} \vec{G}(t,x,s,y)\vec{g}(y)\,dy.
\end{equation}
Similarly, for $\vec{g}\in L^2(\Omega)^N$, let $\vec{u}$
be the unique weak solution in $\rV^{1,0}_2((-\infty,t)\times\Omega)^N$ of
the backward problem \eqref{eq3.15}.
As in \eqref{eq3.16}, $\vec{u}$ has the following representation:
\begin{equation} \label{eq3.16p}
\vec{u}(s,y)=
\int_{\Omega} {}^t\!\vec{G}(s,y,t,x)\vec{g}(x)\,dx.
\end{equation}

We need the following lemmas to prove \eqref{eq5.23} of Theorem~\ref{thm1}.
\begin{lemma}  \label{lem4.1}
Assume that $\vec{g}\in L^2(\Omega)^N$ is supported in a closed set
$F\subset\Omega$.
Let $\vec{u}$ be the weak solution in
$\rV^{1,0}_2((s,\infty)\times\Omega)^N$ of the problem \eqref{eq5.88}.
Then,
\begin{equation}  \label{gaffney}
\int_E \abs{\vec{u}(t,\cdot)}^2 \leq e^{-c\dist(E,F)^2/(t-s)}
\int_F \abs{\vec{g}}^2\quad\forall E\subseteq \Omega,
\end{equation}
where $\dist(E,F)=\inf\set{|x-y|: x\in E, y\in F}$ and
$c=\lambda/(2\Lambda^2)$.
\end{lemma}
\begin{proof}
We may assume that $\dist(E,F)>0$; otherwise \eqref{gaffney} is an
immediate consequence of the energy inequality.
Let $\psi$ be a bounded Lipschitz function on $\Omega$ satisfying
$\abs{D\psi}\le \gamma$ a.e. for some $\gamma>0$ to be fixed later.
Denote
\begin{equation*}
I(t):=\int_{\Omega} e^{2\psi}\abs{\vec{u}(t,\cdot)}^2\quad\forall t\ge s.
\end{equation*}
By following \cite[\S III.4]{LSU}, it is not hard to see that $I$
is absolutely continuous on $[s,\infty)$
and that $I'(t)$ satisfies for a.e. $t>s$
\begin{align*}
I'(t) &=-2\int_\Omega
A^{\alpha\beta}_{ij}D_\beta u^j(t,\cdot) D_\alpha(e^{2\psi}u^i(t,\cdot))\\
&=-2\int_\Omega
e^{2\psi}A^{\alpha\beta}_{ij}D_\beta u^j(t,\cdot) D_\alpha u^i(t,\cdot)
-4\int_\Omega
e^{2\psi}A^{\alpha\beta}_{ij}D_\beta u^j(t,\cdot) D_\alpha\psi\,u^i(t,\cdot)\\
&\leq -2\lambda\int_{\Omega} e^{2\psi}\abs{D\vec{u}(t,\cdot)}^2+4\Lambda\gamma
\int_\Omega e^\psi\abs{D\vec{u}(t,\cdot)} e^\psi\abs{\vec{u}(t,\cdot)}\\
&\leq(2\Lambda^2/\lambda)\gamma^2
\int_\Omega e^{2\psi}\abs{\vec{u}(t,\cdot)}^2
=(2\Lambda^2/\lambda)\gamma^2 I(t).
\end{align*}
The above differential inequality yields
\begin{equation}  \label{eqx001}
I(t)\leq e^{(2\Lambda^2/\lambda)\gamma^2(t-s)}
\|e^\psi\vec{g}\|_{L^2(F)}^2.
\end{equation}
Note that since $F$ is closed, the function
$\dist(x,F)=\min\set{|x-y|:y\in F}$ is a Lipschitz function with
Lipschitz constant $1$ and $\dist(E,F)=\inf_{x\in E} \dist(x,F)$.
Therefore, if we set $\psi(x)=\gamma\min\{\dist(x,F),\dist(E,F)\}$,
then by \eqref{eqx001}, we find
\begin{equation*}
\int_E \abs{\vec{u}(t,\cdot)}^2
\leq \exp\{(2\Lambda^2/\lambda)\gamma^2(t-s)-2\gamma\dist(E,F)\}
\int_F \abs{\vec{g}}^2.
\end{equation*}
The lemma follows if we set
$\gamma=\dist(E,F)/\{(2\Lambda^2/\lambda)(t-s)\}$.
\end{proof}

\begin{lemma}  \label{lem4.2}
Assume that the operator $\cL$ satisfies the property (PH).
Let $\vec{u}$ be the weak solution in
$\rV^{1,0}_2((s,\infty)\times\Omega)^N$ of the problem \eqref{eq5.88}
with $\vec{g}\in L^\infty_c(\Omega)^N$.
Then,
\begin{equation*}    
\abs{\vec{u}(t,x)}\leq C \norm{\vec{g}}_{L^\infty(\Omega)}
\quad\text{if } \sqrt{t-s}<\bar{d}_X,
\end{equation*}
where $C=C(n,N,\lambda,\Lambda,\mu_0,C_0)>0$.
\end{lemma}
\begin{proof}
Let $\rho:=\sqrt{t-s}$ and assume $\rho<\bar{d}_X$.
Denote $A_0=B_{\rho}(x)$ and
$A_k=\set{y\in\Omega: 2^{k-1}\rho\leq |y-x|<2^k\rho}$ for $k\ge 1$.
Since $\vec{g}$ is compactly supported in $\Omega$, we see that
$\vec{g}=\sum_{k=0}^{k_0}\vec{g} 1_{A_k}$ for some $k_0<\infty$.
Denote
\begin{equation*}
\vec{u}_k(t,x)=\int_{A_k} \vec{G}(t,x,s,y)\vec{g}(y)\,dy.
\end{equation*}
Then, it follows from \eqref{eqn:3.58p} that
$\vec{u}=\sum_{k=0}^{k_0} \vec{u}_k$ and that each $\vec{u}_k$ is
a unique weak solution in $\rV^{1,0}_2((s,\infty)\times\Omega)^N$
of the problem \eqref{eq5.88}
with $\vec{g}1_{A_k}$ in place of $\vec{g}$.
If we apply Lemma~\ref{lem4.1} to $\vec{u}_k$ with $E=B_\rho(x)$ and
$F=\overline A_k$, we obtain
\begin{equation*}
\int_{B_\rho(x)}\abs{\vec{u}_k(\tau,y)}^2\,dy \leq
Ce^{-c 2^k}2^{kn}\rho^n\norm{\vec{g}}_{L^\infty(\Omega)}^2
\quad\forall \tau\in (s,t).
\end{equation*}
Therefore, by Lemma~\ref{lem3} we estimate
\begin{equation*}
|\vec{u}(t,x)|\leq \sum_{k=0}^{k_0}|\vec{u}_k(t,x)|\leq
C \sum_{k=1}^\infty e^{-c 2^k}2^{kn/2} \norm{\vec{g}}_{L^\infty(\Omega)}
\leq C \norm{\vec{g}}_{L^\infty(\Omega)}.
\end{equation*}
The lemma is proved.
\end{proof}

\begin{lemma}  \label{lem4.3}
Let $x_0\in\Omega$ and $R<\frac{1}{2}\min(\dist(x_0,\partial\Omega),R_c)$
be given.
Assume that $\eta\in C^\infty_c(B_{2R}(x_0))$  satisfies
$0\leq \eta \leq 1$, $\eta\equiv 1$ in $B_R(x_0)$, and $\abs{D\eta}\leq C/R$.
Then,
\begin{equation*}
\lim_{t\to s} \int_\Omega \vec{G}(t,x,s,y) \eta(y)\,dy = \vec I
\quad \forall x\in B_R(x_0),
\end{equation*}
where $\vec G(t,x,s,y)$ is the Green's matrix of $\cL$ as constructed above
and $\vec I$ is the $N$ by $N$ identity matrix.
\end{lemma}
\begin{proof}
First, we claim that the following identity holds for all $s_1<t$:
\begin{align}
\label{eq3.431p}
\phi^k(X)&=\int_{\Omega}{}^t\!G_{ik}(s_1,\cdot,t,x)\phi^i(s_1,\cdot)
+\int_{s_1}^t\!\int_{\Omega}{}^t\!G_{ik}(\cdot,X)\phi^i_t\\
\nonumber
&\qquad+\int_{s_1}^t\!\int_{\Omega}
{}^t\!A^{\alpha\beta}_{ij}D_\beta {}^t\!G_{jk}(\cdot,X) D_\alpha \phi^i
\quad \forall \vec{\phi}\in C^\infty_{c,p}(\U)^N.
\end{align}
Indeed, by \eqref{eq2.241p}, we have for all $s_1<t$,
\begin{align}
\label{eq3.440p}
\dashint_{\U_{\sigma_\nu}^+(X)}\phi^k&=
\int_\Omega{}^t\!G^{\sigma_\nu}_{ik}(s_1,\cdot,t,x)\phi^i(s_1,\cdot)
+\int_{s_1}^\infty\!\!\!\int_\Omega {}^t\!G^{\sigma_\nu}_{ik}(\cdot,X)\phi^i_t\\
\nonumber
&\qquad+\int_{s_1}^\infty\!\!\!\int_\Omega {}^t\!A^{\alpha\beta}_{ij}
D_\beta{}^t\!G^{\sigma_\nu}_{jk}(\cdot,X)D_\alpha\phi^i\quad
\forall \vec{\phi}\in C^\infty_{c,p}(\U)^N.
\end{align}
Note that by estimates similar to \eqref{eq3.420p} and \eqref{eq3.6.0p},
we have
\begin{equation*}
\lim_{\nu\to \infty}
\int_{\Omega}{}^t\!G^{\sigma_\nu}_{ik}(s_1,\cdot,t,x)\phi^i(s_1,\cdot)
=\int_{\Omega}{}^t\!G_{ik}(s_1,\cdot,t,x)\phi^i(s_1,\cdot).
\end{equation*}
Then, \eqref{eq3.431p} follows from a similar argument
as used in the  proof of \eqref{eq5.22} in Section~\ref{sec4.2} and 
\eqref{eq11.55p}.
If we set $\vec{\phi}(Y)=\eta(y)\vec{e}_l$
and $s_1=s<t$ in \eqref{eq3.431p},
then
\begin{align}
\nonumber
\delta_{kl}&=\int_{\Omega}{}^t\!G_{lk}(s,\cdot,t,x) \eta
+\int_s^t\!\!\int_\Omega {}^t\!A^{\alpha\beta}_{lj}D_\beta {}^t\!G_{jk}(\cdot,X)
D_\alpha\eta\\
\label{eq387p}
&=\int_{\Omega}G_{kl}(t,x,s,\cdot) \eta
+\int_s^t\!\!\int_\Omega {}^t\!A^{\alpha\beta}_{lj}D_\beta {}^t\!G_{jk}(\cdot,X)
D_\alpha\eta=:I+II.
\end{align}
By using H\"older's inequality and \eqref{eqn:G-34}, we estimate
\begin{align*}
\abs{II}&\le C R^{n/2-1}(t-s)^{1/2}
\left(\int_s^t\!\!\!\int_{B_R(x_0)\setminus B_{R/2}(x_0)}
|D{}^t\!\vec G(\cdot,X_0)|^2\right)^{1/2}\\
&\le C R^{-1}(t-s)^{1/2}\qquad\text{if } \sqrt{t-s}< R.
\end{align*}
Therefore, the lemma follows by taking the limit $t$ to $s$ in \eqref{eq387p}.
\end{proof}

Now, we shall prove \eqref{eq5.23} of Theorem~\ref{thm1}.
Assume that $\vec g\in L^2(\Omega)^N$ and $\vec g$ is continuous at $x_0$.
Let $\vec{u}$ be the weak solution in
$\rV^{1,0}_2((s,\infty)\times\Omega)^N$ of the problem \eqref{eq5.88}.
For a given $\epsilon>0$,
choose $R<\frac{1}{2}\min(\dist(x_0,\partial\Omega),R_c)$ such that
$\abs{\vec{g}(x)-\vec{g}(x_0)}< \epsilon/2C$
for all $x$ satisfying $\abs{x-x_0}<2R$, where $C$ is the constant
as appears in Lemma~\ref{lem4.2}.
Let $\eta$ be given as in Lemma~\ref{lem4.3} and
let $\vec{u}_0$, $\vec{u}_\epsilon$, and $\vec{u}_\infty$
be the weak solutions in
$\rV^{1,0}_2((s,\infty)\times\Omega)^N$ of the problem \eqref{eq5.88} with
$\eta\vec{g}(x_0)$, $\eta(\vec{g}-\vec{g}(x_0))$, and $(1-\eta)\vec{g}$
in place of $\vec{g}$, respectively.
Then, by the uniqueness, we have
$\vec{u}=\vec{u}_0+\vec{u}_\epsilon+\vec{u}_\infty$.
First, note that by \eqref{eqn:3.58p}, $\vec{u}_0$ has the representation
\begin{equation}  \label{eq002y}
\vec{u}_0(t,x)=\left(\int_{\Omega} \vec{G}(t,x,s,y)\eta(y)\,dy\right)
\vec{g}(x_0).
\end{equation}
Next, we apply Lemma~\ref{lem4.1} to $\vec{u}_\infty$
with $E=B_{R/2}(x_0)$ and $F=\Omega\setminus B_R(x_0)$ to find
\begin{equation*}
\int_{B_{R/2}(x_0)}\abs{\vec{u}_\infty(\tau,y)}^2\,dy \leq
e^{-cR^2/(\tau-s)}
\norm{\vec{g}}_{L^2(\Omega)}^2\quad \forall \tau>s.
\end{equation*}
Therefore, if $\rho:=\sqrt{t-s}/2<R/4$, then by Lemma~\ref{lem3},
we find
\begin{equation}  \label{eq001y}
\abs{\vec{u}_\infty(t,x)}
\leq C \rho^{-n/2} e^{-c(R/\rho)^2} \norm{\vec{g}}_{L^2(\Omega)}
\quad\forall x\in B_\rho(x_0).
\end{equation}
Finally, we estimate $\vec{u}_\epsilon$ by using Lemma~\ref{lem4.2}:
\begin{equation}  \label{eq003y}
\abs{\vec{u}_\epsilon(t,x)}\leq  \epsilon/2
\quad\text{if } \sqrt{t-s}<\bar{d}_X.
\end{equation}
Combining \eqref{eq002y}, \eqref{eq001y}, and \eqref{eq003y}, we
see that if $\sqrt{t-s}$ is sufficiently small, then
for all $x\in B_\rho(x_0)$, we have $\abs{\vec{u}(t,x)-\vec{g}(x_0)}<\epsilon$.
This completes the proof.

\subsection{Conclusion}
The proof of representation formulas \eqref{eqn:E-70} and \eqref{eq5.24}
given in Section~\ref{sec3.5} as well as the proof of the uniqueness
given in Section~\ref{sec:3.6} also works for general domains
$\Omega$ and $R_c\in (0,\infty]$.
Therefore, Theorem~\ref{thm1} is now proved.

\mysection{Proof of Theorem~\ref{thm2}}   \label{sec5}
\subsection{Proof of the Gaussian bound \eqref{eq5.27b}} \label{sec5.1}
Here, we consider the case $R_c=\infty$ and prove \eqref{eq5.27b}.
In \cite{HofKim2}, by following methods of Davies \cite{Davies} and
Fabes-Stroock \cite{FS}, Hofmann and Kim derived the upper Gaussian bound
of Aronson \cite{Aronson} under a further qualitative assumption
that the coefficients of $\cL$ are smooth.
For the reader's convenience, we reproduce their argument here,
dropping the technical assumption that the coefficients are smooth.

Let $\psi$ be a bounded Lipschitz function on $\bR^n$ satisfying
$\abs{D \psi} \le \gamma$ a.e. for some $\gamma>0$ to be chosen later.
For $t>s$, we define an operator $P^\psi_{s\to t}$ on $L^2(\bR^n)^N$ as follows.
For a given $\vec f\in L^2(\bR^n)^N$, let $\vec{u}$ be the weak solution
in $\rV^{1,0}_2((s,\infty)\times\bR^n)^N$ of the problem
\begin{equation} \label{eq3.6.1}
\left\{\begin{array}{l l}
\cL \vec{u}=0\\
\vec{u}(s,\cdot)= e^{-\psi}\vec{f}.\end{array}\right.
\end{equation}
Then, we define $P^\psi_{s\to t}\vec{f}(x):= e^{\psi(x)}\vec{u}(t,x)$.
Note that it follows from \eqref{eq5.24} that
\begin{equation}  \label{eq3.60.3}
P^\psi_{s\to t}\vec{f}(x)=
e^{\psi(x)}\int_{\bR^n} \vec{\Gamma}(t,x,s,y)e^{-\psi(y)}\vec{f}(y)\,dy
\quad\forall \vec{f}\in L^2(\bR^n)^N.
\end{equation}
We denote
$I(t):=\|e^\psi\vec{u}(t,\cdot)\|_{L^2(\bR^n)}^2$ for all $t\ge s$.
Then, as in the proof of Lemma~\ref{lem4.1},
we find that $I$ is absolutely continuous and satisfies for a.e. $t>s$,
\begin{equation*}
I'(t)\leq (2\Lambda^2/\lambda)\gamma^2 I(t).
\end{equation*}
The above differential inequality with the initial condition
$I(s)=\norm{\vec{f}}^2_{L^2(\bR^n)}$ yields
\begin{equation*}
I(t)\leq e^{(2\Lambda^2/\lambda)\gamma^2(t-s)}\norm{\vec{f}}_{L^2(\bR^n)}^2
\quad \forall t \ge s.
\end{equation*}
Since $I(t)=\|P^\psi_{s\to t}\vec{f}\|_{L^2(\bR^n)}^2$ for $t>s$,
we have derived
\begin{equation} \label{eq3.70}
\|P^\psi_{s\to t}\vec{f}\|_{L^2(\bR^n)}
\leq e^{\nu\gamma^2(t-s)}\norm{\vec{f}}_{L^2(\bR^n)}
\quad \forall t>s,
\end{equation}
where $\nu=\Lambda^2/\lambda$.
By Lemma~\ref{lem3}, we estimate
\begin{align*}
e^{-2\psi(x)}|P^\psi_{s\to t}\vec{f}(x)|^2 &= \abs{\vec{u}(t,x)}^2\\
&\le \frac{C}{(t-s)^{n/2+1}}
\int_s^t\int_{B_{\sqrt{t-s}}(x)}\abs{\vec{u}(\tau,y)}^2\,dy\,d\tau\\
&\le \frac{C}{(t-s)^{n/2+1}} \int_s^t\int_{B_{\sqrt{t-s}}(x)}e^{-2\psi(y)}
|P^\psi_{s\to\tau}\vec{f}(y)|^2\,dy\,d\tau.
\end{align*}
Hence, by using \eqref{eq3.70} we find
\begin{align*}
|P^\psi_{s\to t}\vec{f}(x)|^2&\le C(t-s)^{-n/2-1}
\int_s^t\int_{B_{\sqrt{t-s}}(x)}e^{2\psi(x)-2\psi(y)}
|P^\psi_{s\to\tau}\vec{f}(y)|^2\,dy\,d\tau\\
&\le C(t-s)^{-n/2-1} \int_s^t\int_{B_{\sqrt{t-s}}(x)}e^{2\gamma\sqrt{t-s}}
\abs{P^\psi_{s\to\tau}\vec{f}(y)}^2\,dy\,d\tau\\
&\le C (t-s)^{-n/2-1}\,e^{2\gamma\sqrt{t-s}}
\int_s^t e^{2\nu\gamma^2(\tau-s)}\norm{\vec{f}}_{L^2(\bR^n)}^2\,d\tau\\
&\le C (t-s)^{-n/2}\,e^{2\gamma\sqrt{t-s}+2\nu\gamma^2(t-s)}
\norm{\vec{f}}_{L^2(\bR^n)}^2.
\end{align*}
We have thus derived the following $L^2\to L^\infty$ estimate
for $P^\psi_{s\to t}$:
\begin{equation} \label{eq3.70.3}
\|P^\psi_{s\to t}\vec{f}\|_{L^\infty(\bR^n)}
\leq  C (t-s)^{-n/4}\,e^{\gamma\sqrt{t-s}+\nu\gamma^2(t-s)}
\norm{\vec{f}}_{L^2(\bR^n)} \quad \forall t>s.
\end{equation}
We also define the operator $Q^\psi_{t\to s}$ on $L^2(\bR^n)^N$ for $s<t$ by
setting $Q^\psi_{t\to s}\vec{g}(y)= e^{-\psi(y)}\vec{v}(s,y)$, where
$\vec{v}$ is the weak solution in
$\rV^{1,0}_2((-\infty,t)\times\bR^n)^N$ of the backward problem
\begin{equation} \label{eq3.6.11}
\left\{\begin{array}{l l}
\cLt \vec{v}=0\\
\vec{v}(t,\cdot)= e^{\psi}\vec{g}.\end{array}\right.
\end{equation}
Then, by \eqref{eq3.16}, we find that
\begin{equation*}
Q^\psi_{t\to s}\vec{g}(y)=
e^{-\psi(y)}\int_{\bR^n} {}^t\!\vec{\Gamma}(s,y,t,x)e^{\psi(x)}\vec{g}(x)\,dx
\quad \forall \vec{g}\in L^2(\bR^n)^N.
\end{equation*}
By a similar calculation that leads to \eqref{eq3.70.3}, we derive
\begin{equation} \label{eq3.73.3}
\|Q^\psi_{t\to s}\vec{g}\|_{L^\infty(\bR^n)}
\leq  C (t-s)^{-n/4}\,e^{\gamma\sqrt{t-s}+\nu\gamma^2(t-s)}
\norm{\vec{g}}_{L^2(\bR^n)}
\quad \forall s<t.
\end{equation}
Notice that it follows from \eqref{eq3.6.1} and \eqref{eq3.6.11} that
(cf. \eqref{eq2.17} and \eqref{eqn:3.56})
\begin{equation*}
\int_{\bR^n}(P^\psi_{s\to t}\vec{f}) \cdot \vec{g}=
\int_{\bR^n}\vec{f}\cdot (Q^\psi_{t\to s} \vec{g}).
\end{equation*}
Therefore, by duality, \eqref{eq3.73.3} implies that for
any $\vec{f}\in L^\infty_c(\bR^n)^N$, we have
\begin{equation} \label{eq3.74.3}
\|P^\psi_{s\to t}\vec{f}\|_{L^2(\bR^n)}
\leq  C (t-s)^{-n/4}\,e^{\gamma\sqrt{t-s}+\nu\gamma^2(t-s)}
\norm{\vec{f}}_{L^1(\bR^n)}
\quad \forall t>s.
\end{equation}
Now,  set $r=(s+t)/2$ and observe that by the uniqueness, we have
\begin{equation*}
P^\psi_{s\to t}\vec{f}= P^\psi_{r\to t}(P^\psi_{s\to r}\vec{f})
\qquad \forall \vec{f}\in L^\infty_c(\bR^n)^N.
\end{equation*}
Then, by noting that $t-r=r-s=(t-s)/2$, we obtain from \eqref{eq3.70.3}
and \eqref{eq3.74.3} that for any $\vec{f}\in L^\infty_c(\bR^n)^N$,
we have
\begin{equation*}
\|P^\psi_{s\to t}\vec{f}\|_{L^\infty(\bR^n)}
\leq  C (t-s)^{-n/2}\,e^{\gamma\sqrt{2(t-s)}+\nu\gamma^2(t-s)}
\norm{\vec{f}}_{L^1(\bR^n)}
\quad \forall t>s.
\end{equation*}
For fixed $x, y\in \bR^n$ with $x\neq y$,
the above estimate and \eqref{eq3.60.3} imply, by duality,
\begin{equation} \label{eq3.83}
e^{\psi(x)-\psi(y)}\abs{\vec{\Gamma}(t,x,s,y)}_{op}
\leq  C (t-s)^{-n/2}\,e^{\gamma\sqrt{2(t-s)}+\nu\gamma^2(t-s)}
\quad \forall t>s.
\end{equation}
Define $\psi_0$ on $[0,\infty)$ by setting
\begin{equation*}
\left\{\begin{array}{l l}
\psi_0(r)=r&\text{ if } r \leq \abs{x-y} \\
\psi_0(r)=\abs{x-y}&\text{ if } r>\abs{x-y}.
\end{array}\right.
\end{equation*}
Let $\psi(z):=\gamma\psi_0(|z-y|)$ where $\gamma=\abs{x-y}/2\nu(t-s)$.
Then, $\psi$ is a bounded Lipschitz function satisfying
$\abs{D\psi} \leq \gamma$ a.e., and thus \eqref{eq3.83} yields
\begin{equation*}  
\abs{\vec{\Gamma}(t,x,s,y)}_{op}
\leq  C (t-s)^{-n/2}\,
\exp\{\xi/\sqrt{2}\nu-\xi^2/4\nu\},
\end{equation*}
where $\xi=\abs{x-y}/\sqrt{t-s}$.
Let $C_0=C_0(\nu)=C_0(\Lambda^2/\lambda)$ be chosen so that
\begin{eqnarray*}
\exp(\xi/\sqrt{2}\nu-\xi^2/4\nu)\le C_0\exp(-\xi^2/8\nu)
\quad\forall \xi\in [0,\infty).
\end{eqnarray*}
If we set $\kappa=1/8\nu=\lambda/8\Lambda^2$, then we obtain
\begin{equation*}    
\abs{\vec{\Gamma}(t,x,s,y)}_{op}\le
C(t-s)^{-n/2}\exp\left\{-\kappa|x-y|^2/(t-s)\right\},
\end{equation*}
where $C=C(n,N,\lambda,\Lambda,\mu_0,C_0)>0$.
We have proved \eqref{eq5.27b}.
\subsection{Proof of \eqref{eq5.14z} and \eqref{eq5.27}} \label{sec:5-2}
We begin by proving the identity \eqref{eq5.14z}.
First note that \eqref{eq5.21}, \eqref{eq5.15}, and \eqref{eq5.15b} imply
that both $\abs{\vec\Gamma(t,x,r,\cdot)}$ and $\abs{\vec\Gamma(r,\cdot,s,y)}$
belong to $L^2(\bR^n)$ for all $x,y\in\bR^n$ and $r\in (s,t)$.
Then, by the uniqueness (cf. \eqref{eq5.24}) and Fubini's theorem,
we have for all $\vec g\in L^\infty_c(\bR^n)^N$,
\begin{align*}
\int_{\bR^n}\vec\Gamma(t,x,s,y)\vec{g}(y)\,dy
&=\int_{\bR^n}\vec\Gamma(t,x,r,z)\left(
\int_{\bR^n}\vec\Gamma(r,z,s,y)\vec{g}(y)\,dy\right)dz\\
&=\int_{\bR^n}\left(\int_{\bR^n} \vec\Gamma(t,x,r,z)\vec\Gamma(r,z,s,y)
\,dz\right) \vec{g}(y) \,dy.
\end{align*}
Since $\vec g\in L^\infty_c(\bR^n)^N$ is arbitrary, we conclude \eqref{eq5.14z}.

We now turn to the proof of the pointwise bound \eqref{eq5.27}.
By following the proof of \eqref{eq5.27b} in Section~\ref{sec5.1},
it is routine to check
\begin{equation}    \label{eq5.13z}
\abs{\vec{\Gamma}(t,x,s,y)}_{op}\le
C(t-s)^{-n/2}e^{-\kappa|x-y|^2/(t-s)}
\quad\text{if }0<t-s<R_c^2.
\end{equation}
Next, recall the identity
\begin{equation}    \label{eq5.15z}
\Phi(t+s,x-y)=\int_{\bR^n} \Phi(t,x-z)\Phi(s,z-y)\,dz
\quad\forall x,y\in\bR^n \quad\forall s,t>0,
\end{equation}
where $\Phi(t,x)=(4\pi t)^{-n/2}\exp\{-|x|^2/4t\}$.
Note that \eqref{eq5.13z} implies that
\begin{equation}    \label{eq5.16z}
\abs{\vec{\Gamma}(t,x,s,y)}_{op}\leq C(\pi/\kappa)^{n/2}
\Phi((t-s)/4\kappa,x-y)
\quad\text{if } 0<t-s\leq R_c^2.
\end{equation}
Let $\ell$ be the largest integer that is strictly less than $(t-s)/R_c^2$.
Denote $t_j=s+j R_c^2$ for $j=0,\ldots, \ell$ and $t_{\ell+1}=t$
so that $t_0=s$ and $t_\ell<t=t_{\ell+1}$.
Note that $|t_j-t_{j-1}|\leq R_c^2$ for all $j=1,\ldots,\ell+1$.
Then, by \eqref{eq5.14z}, \eqref{eq5.15z}, and \eqref{eq5.16z}, we have
\begin{align*}
\abs{\vec\Gamma(t,z_{\ell+1},s,z_0)}_{op}
&\leq \int_{\bR^n}\!\!\!\cdots\int_{\bR^n}
\prod_{j=1}^{\ell+1}
\abs{\vec\Gamma(t_j,z_j,t_{j-1},z_{j-1})}_{op}\,dz_1\cdots\,dz_\ell\\
&\leq\{C(\pi/\kappa)^{n/2}\}^{\ell+1} \Phi((t-s)/4\kappa,z_{\ell+1}-z_0)\\
&=C^{\ell+1}(\pi/\kappa)^{n\ell/2}(t-s)^{-n/2}
\exp\{-\kappa|z_{\ell+1}-z_0|^2/(t-s)\}.
\end{align*}
Therefore, \eqref{eq5.27} follows if we set $z_0=y$, $z_{\ell+1}=x$,
and $\gamma=\ln\{C(\pi/\kappa)^{n/2}\}$.

\subsection{Proof of the identity \eqref{eq5.99}}
First, note that as in \eqref{eq3.431p}, the following identity holds
for all $s_1<t$:
\begin{align}
\label{eq3.431}
\phi^k(X)&=\int_{\bR^n}{}^t\!\Gamma_{ik}(s_1,\cdot,t,x)\phi^i(s_1,\cdot)
+\int_{s_1}^{t}\!\int_{\bR^{n}}{}^t\!\Gamma_{ik}(\cdot,X)\phi^i_t\\
\nonumber
&\qquad+\int_{s_1}^{t}\!\int_{\bR^{n}}
{}^t\!A^{\alpha\beta}_{ij}D_\beta {}^t\!\Gamma_{jk}(\cdot,X) D_\alpha \phi^i
\quad\forall \vec{\phi}\in C^\infty_{c,p}(\bR^{n+1})^N.
\end{align}
Now, let $\zeta\in C^\infty_c(\bR^n)$ be a cut-off function satisfying
\begin{equation*}
\zeta\equiv 1\text{ in } B_1(0),
\quad \zeta\equiv 0\text{ outside } B_2(0),\quad
0\leq \zeta \leq 1,\quad \abs{D\zeta}\leq 2.
\end{equation*}
If we set $\vec{\phi}(Y)=\zeta((y-x)/R)\vec{e}_l$
and $s_1=s<t$ in \eqref{eq3.431},
then
\begin{align}
\label{eq387}
\delta_{kl}&=
\int_{\bR^n}{}^t\!\Gamma_{lk}(s,\cdot,t,x) \zeta((\cdot-x)/R)
+\int_s^t\!\!\!\int_{\bR^{n}}
{}^t\!A^{\alpha\beta}_{ij}D_\beta {}^t\!\Gamma_{jk}(\cdot,X) D_\alpha\phi^i\\
\nonumber
&=I+II.
\end{align}
By \eqref{eqn:E-66}, \eqref{eq5.27}, and the dominated convergence theorem,
we have
\begin{equation*}
\lim_{R\to\infty}I=\lim_{R\to\infty}
\int_{\bR^n}\Gamma_{kl}(t,x,s,y)\zeta((y-x)/R)\,dy=
\int_{\bR^n}\Gamma_{kl}(t,x,s,y)\,dy.
\end{equation*}
On the other hand, using H\"older's inequality we estimate $II$ by
\begin{equation*}
\abs{II}\le
C R^{n/2-1}(t-s)^{1/2}
\left(\int_s^t\!\!\!\int_{B_{2R}(x)\setminus B_R(x)}
\abs{D{}^t\!\vec{\Gamma}(\cdot,X)}^2\right)^{1/2}.
\end{equation*}
Next, let $\eta\in C^\infty_c(B_{3R}(x))$ be such that $\eta\equiv 1$ in
$B_{2R}(x)\setminus B_R(x)$, $\eta\equiv 0$ in $B_{R/2}(x)$,
and $\abs{D \eta}\leq C/R$.
From \eqref{eq5.27} and \eqref{eqn:E-66p}, it follows that
$\lim_{s \uparrow t}\eta{}^t\!\vec{\Gamma}(s,\cdot,t,x)\equiv 0$.
Note that ${}^t\!\vec{\Gamma}(\cdot,X)$ satisfies $\cLt\vec u=0$ in
$\set{(s,y)\in\bR^{n+1}: s<t}$. Therefore, as it is done in the proof
of \eqref{eq5.59}, we derive
\begin{equation*}
\int_s^t\!\!\!\int_{B_{2R}(x)\setminus B_R(x)}
\abs{D{}^t\!\vec{\Gamma}(\cdot,X)}^2
\leq C R^{-2}\int_s^t\!\!\!\int_{B_{3R}(x)\setminus B_{R/2}(x)}
\abs{{}^t\!\vec{\Gamma}(\cdot,X)}^2.
\end{equation*}
If $R>\sqrt{t-s}$, then by \eqref{eqn:E-66p} and \eqref{eq5.27}, we estimate
(cf. \eqref{eq5.59})
\begin{equation*}
\int_s^t\!\!\!\int_{B_{3R}(x)\setminus B_{R/2}(x)}
\abs{{}^t\!\vec{\Gamma}(\cdot,X)}^2\leq C e^{2\gamma (t-s)/{R_c^2}} R^{2-n}.
\end{equation*}
Therefore, we find that $II$ is bounded by
\begin{equation*}
\abs{II}\le
C R^{-1}(t-s)^{1/2}e^{\gamma(t-s)/R_c^2}\quad\text{if }R>\sqrt{t-s},
\end{equation*}
and thus, we obtain \eqref{eq5.99} by taking $R\to\infty$ in \eqref{eq387}.

\mysection{Appendix} \label{appendix}
\begin{lemma}\label{lem:Ap01}
Let $\set{u_k}_{k=1}^\infty$ be a sequence in $V_2(\U)$ (resp. $\rV_2(\U)$),
where $\U=\bR\times\Omega$ and $\Omega$ is an open connected set in $\bR^n$.
If $\sup_k\tri{u_k}_\U=M<\infty$, then there exist
a subsequence $\{u_{k_j}\}_{j=1}^\infty\subseteq \{u_k\}_{k=1}^\infty$
and $u\in V_2(\U)$ (resp. $\rV_2(\U)$) with $\tri{u}_{\U}\leq M$
such that $u_{k_j}\wto u$ ``very weakly'' in $V_2(\U)$.
\end{lemma}
\begin{proof}
Let $\U_T:=\U \cap (-T,T)\times\bR^n$, where $T>0$.
We claim that for all $T>0$, there exist
a subsequence $\{u_{k_j}\}_{j=1}^\infty\subseteq \{u_k\}_{k=1}^\infty$
and $u\in V_2(\U_T)$ with $\tri{u}_{\U_T}\leq M$
such that $u_{k_j}\wto u$ weakly in $W^{1,0}_2(\U_T)$.
It will be clear from the proof below that
if $\set{u_k}_{k=1}^\infty$ is a sequence in $\rV_2(\U)$,
then we also have $u\in \rV_2(\U_T)$.
Therefore, once we prove the claim, the lemma will follow
from a standard diagonalization argument.

Now, let us prove the claim.
Notice that $\norm{u_k}_{W^{1,0}_2(\U_T)}\leq C \tri{u_k}_{\U_T}\leq C M$
for some $C=C(n,N,T)<\infty$.
Therefore, there exist a subsequence
$\{u_{k_j}\}_{j=1}^\infty\subseteq \{u_k\}_{k=1}^\infty$
and $u\in W^{1,0}_2(\U_T)$ such that
such that $u_{k_j}\wto u$ weakly in $W^{1,0}_2(\U_T)$.
For simplicity, let us relabel the subsequence and assume
that $u_j\wto u$ in $W^{1,0}_2(\U_T)$. 
We only need to show that $u\in V_2(\U_T)$.
Let $t\in (-T,T)$ be fixed and choose $h>0$ so small that 
$(t-h,t+h)\subset (-T,T)$.
Since $u_j\wto u$ weakly in $L^2(\U_T)$,
we also have $u_j\wto u$ weakly in $L^2((t-h,t+h)\times\Omega)$, and thus
\begin{equation*}
\frac{1}{2h}\int_{t-h}^{t+h}\!\!\!\int_\Omega\abs{u(\tau,x)}^2\,dx\,d\tau
\leq \liminf_{j\to\infty}
\frac{1}{2h}\int_{t-h}^{t+h}\!\!\!\int_\Omega\abs{u_j(\tau,x)}^2\,dx\,d\tau
\leq M^2.
\end{equation*}
Therefore, by taking $h\to 0$, we obtain
$\esssup_{t\in(-T,T)}\|u(t,\cdot)\|_{L^2}^2\leq M^2$, which together with
$u\in W^{1,0}_2(\U_T)$ implies that $u\in V_2(\U_T)$.
It is clear that
$\tri{u}_{\U_T}\leq \sup_k\tri{u_k}_{\U_T}\leq M$.
We have proved the claim, and thus the lemma.
\end{proof}

\begin{lemma}\label{lem:Ap02}
Assume the coefficients $\vec A^{\alpha\beta}$ of the operator
$\cL$ in \eqref{eqn:P-01} are independent of $x$ 
(i.e., $\vec A^{\alpha\beta}=\vec A^{\alpha\beta}(t)$) and satisfy
the conditions \eqref{eqn:P-02} and \eqref{eqn:P-03}.
Then, there exists a constant $C=C(n,N,\lambda,\Lambda)>0$ such that
all weak solutions  $\vec{u}$ of $\cL\vec u =0$ in $Q_R^-(X_0)$ satisfy
\begin{equation}  \label{eqAp01}
\int_{Q_\rho^-(X_0)}\abs{D \vec u}^2\leq
C\left(\frac \rho r\right)^{n+2}
\int_{Q_r^-(X_0)}\abs{D \vec u}^2 \quad \forall 0<\rho<r\leq R
\end{equation}
and similarly,
all weak solutions  $\vec{u}$ of $\cLt\vec u =0$ in $Q_R^+(X_0)$ satisfy
\begin{equation}  \label{eqAp02}
\int_{Q_\rho^+(X_0)}\abs{D \vec u}^2\leq
C\left(\frac \rho r\right)^{n+2}
\int_{Q_r^+(X_0)}\abs{D \vec u}^2 \quad \forall 0<\rho<r\leq R.
\end{equation}
\end{lemma}
\begin{proof}
First, we consider the case $r=1$ and $X_0=0$.
Assume that $\vec u$ is a weak solution of $\cL \vec{u}=0$ in $Q_1^-(0)$.
Let $k$ be the smallest integer strictly larger than $n/2$.
Since the coefficients $\vec{A}^{\alpha\beta}$ of the operator $\cL$
are independent of $x$,
we obtain by iterative applications of the energy inequalities
\begin{equation}  \label{eqAp06}
\norm{\vec u(t,\cdot)}_{W^{k,2}(B_{1/2}(0))}
\leq C\norm{\vec u}_{L^2(Q_1^-(0))}
\quad \forall t \in (-1/4,0),
\end{equation}
where $C=C(n,\lambda,\Lambda)$.
Moreover,
by the Sobolev inequality (see e.g., \cite[\S 5.6.3]{Evans})
\begin{equation}  \label{eqAp03}
\norm{\vec u(t,\cdot)}_{L^\infty(B_{1/2}(0))}
\leq C \norm{\vec u(t,\cdot)}_{W^{k,2}(B_{1/2}(0))}
\quad \forall t\in (-1/4,0),
\end{equation}
where $C=C(n,\lambda,\Lambda,B_{1/2}(0))=C(n,\lambda,\Lambda)$.
Then, we obtain from \eqref{eqAp03} and \eqref{eqAp06}
\begin{equation*}
\int_{Q_\rho^-(0)}\abs{\vec u}^2
\leq C \rho^{n+2}\norm{\vec u}_{L^2(Q_1^-(0))}^2
\quad \forall \rho \in (0,1/2),
\end{equation*}
and thus by replacing $C$ by $\max(2^{n+2}, C)$ if necessary,
we have
\begin{equation}  \label{eqAp04}
\int_{Q_\rho^-(0)}\abs{\vec u}^2
\leq C \rho^{n+2}\norm{\vec u}_{L^2(Q_1^-(0))}^2
\quad \forall \rho \in (0,1).
\end{equation}
Now, we consider general case.
Assume that $\vec u$ is a weak solution of $\cL \vec u=0$ in $Q_r^-(X_0)$.
Then, $\tilde{\vec u}(t,x):=\vec{u}((t-t_0)/r^2,(x-x_0)/r)$ satisfies
$\tilde \cL \tilde{\vec u}=0$ in $Q_1^-(0)$, where the coefficients
$\tilde \cL$ are still independent of $x$ and satisfies \eqref{eqn:P-02}
and \eqref{eqn:P-03}; more precisely, they are given by
$\tilde{\vec A}{}^{\alpha\beta}(t)={\vec A}^{\alpha\beta}((t-t_0)/r^2)$.
Therefore, by \eqref{eqAp04} and the change of variables, we obtain
\begin{equation*}
r^{-n-2}\int_{Q_{\rho r}^-(X_0)}\abs{\vec u}^2
\leq C (\rho/r)^{n+2}\norm{\vec u}_{L^2(Q_r^-(X_0))}^2
\quad \forall \rho \in (0,1).
\end{equation*}
The above estimate is equivalent to
\begin{equation}  \label{eqAp05}
\int_{Q_{\rho}^-(X_0)}\abs{\vec u}^2 \leq
C\left(\frac \rho r\right)^{n+2} \int_{Q_r^-(X_0)}\abs{\vec u}^2
\quad \forall \rho \in (0,r).
\end{equation}
Therefore, \eqref{eqAp01} follows from \eqref{eqAp05} and the observation
that for all $k=1,\ldots,n$, $D_k\vec u$ is also a weak solution of
$\cL \vec u=0$ in $Q_r^-(X_0)$.
\end{proof}

\textbf{Acknowledgments} The authors thank Steve Hofmann and Neil Trudinger for very useful discussion.


\end{document}